\documentclass{amsart}

\usepackage[hyphens]{url}
\usepackage{hyperref}
\usepackage[hyphenbreaks]{breakurl}
\usepackage{amsmath}
\usepackage{amssymb}
\usepackage{mathrsfs}
\usepackage{url}
\usepackage{tikz}
\usepackage{tikz-cd}

\usepackage[square,numbers]{natbib}
\usepackage[colorinlistoftodos]{todonotes}

\theoremstyle{plain}
\newtheorem{thm}{Theorem}[section]
\newtheorem*{thm*}{Theorem}
\newtheorem{lem}[thm]{Lemma}
\newtheorem{prop}[thm]{Proposition}
\newtheorem{cor}[thm]{Corollary}
\newtheorem{quest}[thm]{Question}
\newtheorem*{quest*}{Question}

\theoremstyle{definition}
\newtheorem{defn}[thm]{Definition}
\newtheorem{eg}[thm]{Example}
\newtheorem{noneg}[thm]{Non-example}
\newtheorem{fact}[thm]{Fact}

\theoremstyle{remark}
\newtheorem{rem}[thm]{Remark}

\newenvironment{claim}[1]{\par\noindent\underline{Claim}\space#1}{}
\newenvironment{claimproof}[1]{\par\noindent\underline{Proof of the Claim:}\space#1}{\hfill $\blacksquare$}

\providecommand{\:}{}
\renewcommand{\:}[0]{\ : \ }

\newcommand{\I}[0]{\mathcal{I}}
\newcommand{\U}[0]{\mathcal{U}}

\newcommand{\M}{\mathcal{M}}
\providecommand{\L}{}
\renewcommand{\L}{\mathcal{L}}

\newcommand{\fhi}[0]{\varphi}

\newcommand{\A}{\mathcal{A}}
\renewcommand{\H}{\mathcal{H}}
\newcommand{\C}{\mathcal{C}}
\renewcommand{\I}{\mathcal{I}}

\newcommand{\B}[0]{\mathcal{B}}
\renewcommand{\P}{\mathcal{P}}
\newcommand{\K}{\mathcal{K}}

\renewcommand{\and}{\ \& \ }
\newcommand{\Iff}[0]{\text{ iff }}

\newcommand{\tp}{\mathrm{tp}}
\newcommand{\age}{\mathrm{age}}

\newcommand{\OP}{\mathrm{OP}}
\newcommand{\IP}{\mathrm{IP}}
\newcommand{\NIP}{\mathrm{NIP}}
\newcommand{\SOP}{\mathrm{SOP}}
\newcommand{\NSOP}{\mathrm{NSOP}}
\newcommand{\TP}{\mathrm{TP}}
\newcommand{\NTP}{\mathrm{NTP}}
\newcommand{\SM}{\mathrm{SM}}
\newcommand{\NSM}{\mathrm{NSM}}
\newcommand{\PM}{\mathrm{PM}}
\newcommand{\NPM}{\mathrm{NPM}}
\newcommand{\CM}{\mathrm{CM}}

\newcommand{\UPM}{\mathrm{UPM}}
\newcommand{\NUPM}{\mathrm{NUPM}}

\title[Notions of maximality in first-order theories]{A walk on the wild side:\\ Notions of maximality in first-order theories}

\author{Michele Bailetti}
\date{\today}

\begin{document}
\begin{abstract}
    In the classification of complete first-order theories, many dividing lines have been defined in order to understand the complexity and the behavior of some classes of theories. In this paper, using the concept of patterns of consistency and inconsistency, we describe a general framework to study dividing lines and we introduce a notion of maximal complexity by requesting the presence of all the exhibitable patterns of definable sets. Weakening this notion, we define new properties (Positive Maximality and the $\PM^{(k)}$ hierarchy) and prove some results about them. In particular, we show that $\PM^{(k+1)}$ theories are not $k$-dependent. Moreover, we provide an example of a $\PM$ but $\NSOP_4$ theory (showing that $\SOP$ and the $\SOP_n$ hierarchy, for $n \geq 4$, can not be described by \emph{positive} patterns) and, for each $1<k<\omega$, an example of a $\PM^{(k)}$ but $\NPM^{(k+1)}$ theory (showing that the newly defined hierarchy does not collapse).     
\end{abstract}
\maketitle

\setcounter{tocdepth}{1}
\tableofcontents

\section{Introduction}
Classification theory is a line of research in model theory, started essentially by Shelah, in which we define and study dividing lines that differentiate classes of theories in order to understand their complexity and behavior. One of the first dividing lines, introduced by Shelah, is \emph{stability} (See Definition \ref{classical-div-lines}), defined as a generalization of Morley's notion of \emph{totally transcendental} theories. Stable theories can be characterized by the absence of a local combinatorial property on definable sets (\emph{Order Property}, See Definition \ref{classical-div-lines}, (1)). As shown by Shelah, the presence of this local property (the theory is said to be \emph{unstable} in this case) implies bad structural behaviors of the theory. For instance 
\begin{thm*}[Shelah,\cite{Shelah-numberofnonisomodels}]
    If $T$ is unstable and $\lambda > |T| + \aleph_0$, then $T$ has exactly $2^\lambda$ non-isomorphic models of size $\lambda$. 
\end{thm*}
On the other hand, the assumption of stability provides methods and tools to analyze and understand this class of theories.
The rich framework provided by stability led researchers to try to extend these methodologies outside stability, defining more and more dividing lines and giving structure to the universe of first-order theories. 

Two prototypical examples of unstable theories, dense linear orders without endpoints and the theory of the random graph, led to the definition of two incomparable extensions of the class of stable theories whose intersection characterizes it: $\NIP$ and $\NSOP$ (See Definition \ref{classical-div-lines} and Fact \ref{implications-classical-div-lines}).
In order to study the existence of saturated elementary extensions of models of first-order theories, Shelah (in \cite{SHELAH-simpleunstable}) isolated a class of theories, extending stability, called \emph{simple theories} (omitting the \emph{tree property}, see Definition \ref{classical-div-lines}, (4)). Simple theories were then intensely studied by Hrushovski, Chatzidakis, Pillay, B. Kim et al (See, e.g., \cite{HrushovskiPillay1994GroupsDI,kim-forkingSimpletheories,Kimpillay-fromStabtosimp,KIMPILLAY-SIMPLETHEORIES,zoe_ehud-fieldwithaut}), showing that, with the assumption of omitting the tree property, many typically stability-theoretic tools behave correctly and are useful in this enlarged framework. A plethora of dividing lines has been defined, along with great developments regarding the classes of theories that are characterized by them, making classification theory a rich and exciting line of research in model theory.

Many of these dividing lines are defined, similarly to stability, by the absence of specific configurations of definable sets. In this paper, abstracting the essential and combinatorial nature of these properties, we introduce the concept of \emph{patterns of consistency and inconsistency} (See Definition \ref{def-patterns}) in order to give a general framework for talking about dividing lines. This concept was first defined by Shelah in \cite{Shelah1999ONWI} (Definition 5.17) with a different notation and more recently appeared in \cite{kruckman2023newkim} (Definition 6.2) in a fashion similar to ours.

Patterns essentially describe abstract configurations of finite collections of sets by asserting conditions about the emptiness (or not) of their intersections. A formula exhibits a pattern in a theory if there are parameters in a model of the theory such that the sets defined by the formula and the parameters satisfy the conditions described by the pattern (See Definition \ref{defn-exhibition-of-patterns}). In this sense, we can capture the combinatorial properties that define dividing lines by exhibition of patterns (a similar approach to ours can be found in \cite{Hill-Guingona_Pos-loc-prop}). 
In Section \ref{Sec-patterns}, we define all the notions relative to patterns and describe this general framework for dividing lines. 

On the wild side of every previously considered dividing line, two notions of maximal complexity for a complete first-order theory have been defined: Cooper's property $1$ (See \cite{cooper}) and ``an extreme condition of the form of unstability" called \emph{Straight Maximality} ($\SM$) defined by Shelah (\cite{Shel-modeltheory}). In Section \ref{Sec-Maximality}, we show that these two notions are equivalent and fit in the framework of exhibition of patterns. Indeed, these two properties are characterized by the existence of a formula exhibiting all the \emph{exhibitable} (see Definition \ref{defn-exhibition-of-patterns}) patterns. We kept a name similar to Shelah's for this property (\emph{Straight up Maximality}).

One of the main approaches in classification theory, as we described above, is to extend nice properties and tools on the tame side of some dividing line to a wider class of theories. Here we reverse the approach and start from the worst (in some sense) combinatorial property that a complete first-order theory can have and study how theories behave when we weaken this bad property. Using this methodology, in Section \ref{secWeak} we define weakenings of $\SM$ by requesting the presence of a formula exhibiting a smaller collection of patterns, guided by the examples of previously defined properties of formulas. For instance, both $\OP$ and $\IP$ do not require any \emph{inconsistency conditions}. We thus define Consistency Maximality ($\CM$, see Definition \ref{CM-defn}) by requiring the existence of a formula realizing all the patterns without inconsistency conditions. We show that $\CM$ is equivalent to $\IP$ (even at the level of formulas), giving a different perspective to this well-known property. 
Other dividing lines, such as $\TP$, $\TP_1$ and $\TP_2$ (See Definition \ref{classical-div-lines}, (4),(5),(6)), can be characterized by realizing patterns in which the negation of the formula is not involved. Using this observation, we define a collection of patterns (\emph{positive patterns}, see Definition \ref{positive-patt-defn}) and a weaker version of $\SM$, called \emph{Positive Maximality} ($\PM$, see Definition \ref{PM-defn}) by requesting the existence of a formula exhibiting all the positive patterns. In particular, all the properties that can be characterized by realization of \emph{positive} patterns are implied by $\PM$. Since our characterization of $\SOP$ uses patterns that are not positive, a natural question is the following: 
\begin{quest*}
Is there a positive maximal $\NSOP$ theory?
\end{quest*}
In Section \ref{secWeak}, we answer positively to the above question by finding an example of a $\NSOP_4$ (See Definition \ref{classical-div-lines},($7_4$)) positive maximal theory (See Example \ref{pm-nsop-example}). In particular, this shows that neither $\SOP$ nor $\SOP_n$ for $n \geq 4$ (See Definition \ref{classical-div-lines},($7_n$)) can be characterized through positive patterns. We point out that, in \cite{kaplan2023generic} (Lemma 7.3), Kaplan, Ramsey and Simon proved a characterization of $\SOP_3$ which can be obtained by exhibition of positive patterns.  

In Section \ref{Sec-PMk}, by putting restraints on the size of the inconsistency conditions in positive patterns, we define a collection of properties $(\PM^{(k)})_{1<k<\omega}$ (See Definition \ref{PM_k-defn}), weaker than $\PM$ but still strong enough to imply all the classical positive dividing lines. We show that, at the level of theories, $\PM^{(k+1)}$ implies $\PM^{(k)}$ (See Proposition \ref{PM_k+1-implies_PM_k}) and we give $\NSOP_4$, $\PM^{(k)}$ and $\NPM^{(k+1)}$ examples showing that the above implication is strict (See Example \ref{PM_k,NPM_k+1-examples}). We show that $\PM^{(k)}$ is equivalent to the realization (in some sense) of every finite $k$-hypergraph and can be witnessed by a (generic ordered $k$-hypergraph)-indexed generalized indiscernible. The $k$-independence property ($\IP_k$) is a higher-arity version of $\IP$ introduced by Shelah. In \cite{n-dependence}, Chernikov, Palacin and Takeuchi showed the equivalence of $\NIP_k$ with the collapse of the above kind of generalized indiscernibles. This led the way for showing the following result.
\begin{thm*}[See \ref{PMk+1impliesIPk}]
    For every $0<k<\omega$, if $T$ is $\PM^{(k+1)}$, then $T$ has $\IP_k$. 
\end{thm*}
The notions of maximality introduced in this paper as essentially given by \emph{binary} properties. The above result shows, in particular, that the binary notion of positive maximality implies all the $k$-ary versions of $\IP$ (at each step of the way). We also point out that there are examples of simple $\IP_k$ theories, and thus $\PM^{(k+1)}$ is not equivalent to $\IP_k$.  

Many questions arise about these notions of maximality, the behavior of classes of theories defined by them and their relations with previously defined dividing lines. For instance, which of the standard examples of $\SOP_3$ but $\NSOP$ theories are $\PM$ (or $\PM^{(k)}$ for some $1<k<\omega$)? Is there a $\SOP_3$ but $\NSOP$ and $\NPM^{(2)}$ theory? If not, we would have that $\NSOP \cap \NPM^{(2)} = \NSOP_3$; this would imply the collapse of the $\SOP_n$ hierarchy inside $\NTP_2$.

\section{Preliminaries}\label{Sec-Prem}
Unless otherwise stated, in what follows $T$ is a complete first-order $\L$-theory, $\U \models T$ is a monster model. Whenever we state the existence of some tuples, we mean tuples of elements from $\U$ and every set is a \emph{small} (of cardinality less then the saturation of $\U$) subset of $\U$. Frequently we work with \emph{partitioned formulas} $\fhi(x;y)$ in which the free variables are divided into \emph{object variables} and \emph{parameters variables} (the variables $x$ and $y$ respectively). Moreover, whenever we write $x,y,z,...$ for variables we mean tuples of variables (not necessarily singletons) and, as a notation, we write $\U^x,\U^y,\U^z,...$ for the set of tuples of elements of $\U$ that can be plugged in for $x,y,z,...$ respectively. By convention, particularly when dealing with the class of finite models of a theory in a relational language, we allow the empty structure. This simplifies the proof that a class of finite structures is a Fra\"issé class (see Examples \ref{pm-nsop-example} and \ref{PM_k,NPM_k+1-examples}). We note that the results are easily seen to still be true even when we exclude the empty structure.  

\subsection{Dividing lines}
We recall the definition of some of the properties defining dividing lines in the universe of first-order theories. Most of these properties were introduced by Shelah in \cite{shelah1990classification} and \cite{SHELAH-toward-class-unstable-t}. For the definitions and a bird's-eye view of the universe we refer to \cite{universe}.
\begin{defn}\label{classical-div-lines}
    A partitioned formula $\fhi(x;y)$ has 
    \begin{enumerate}
        \item the \emph{order property} ($\OP$) if there are tuples $(a_ib_i)_{i<\omega}$ such that \[\models \fhi(a_i;b_j) \iff i<j\]
        \item the \emph{independence property} ($\IP$) if there are tuples $(a_i)_{i<\omega}$ and $(b_I)_{I \subseteq \omega}$ such that \[\models \fhi(a_i;b_I) \iff i \in I\]
        \item the \emph{strict-order property} ($\SOP$) if there are tuples $(b_i)_{i<\omega}$ such that 
        \[\models \exists x (\fhi(x;b_j) \land \neg\fhi(x;b_i)) \iff i < j\]
        \item the $k$-\emph{tree property} ($k$-$\TP$) if there exist tuples $(b_\eta)_{\eta \in \omega^{<\omega}}$ such that 
        \begin{itemize}
            \item paths are consistent: for all $f \in \omega^\omega$, we have that $\{\fhi(x;b_{f|i}) \: i <\omega\}$ is consistent, and 
            \item levels are $k$-inconsistent: for all $\eta \in \omega^{<\omega}$, $\{\fhi(x;b_{\eta \frown i}) \: i<\omega\}$ is $k$-inconsistent (i.e. every subset of size $k$ is inconsistent).
        \end{itemize}
        It has the \emph{tree property} ($\TP$) if it has $k$-$\TP$ for some $1<k<\omega$
        \item the \emph{tree property of the first kind} ($\TP_1$) if there are tuples $(b_{\eta})_{\eta \in \omega^{<\omega}}$ such that 
        \begin{itemize}
            \item paths are consistent: for all $f \in \omega^\omega$, we have that $\{\fhi(x;b_{f|i}) \: i<\omega\}$ is consistent, and 
            \item incomparable elements are inconsistent: for all incomparable $\eta,\mu \in \omega^{<\omega}$, we have that $\{\fhi(x;b_\eta),\fhi(x;b_\mu)\}$ is inconsistent.
        \end{itemize}
        \item the $k$-\emph{tree property of the second kind} ($k$-$\TP_2$) if there are tuples $(b_i^j)_{i,j<\omega}$ such that 
        \begin{itemize}
            \item paths are consistent: for every $f \in \omega^\omega$, we have that $\{\fhi(x;b_{i}^{f(i)}) \: i<\omega\}$ is consistent, and 
            \item rows are $k$-inconsistent: for each $i<\omega$, $\{\fhi(x;b_i^j) \: j<\omega\}$ is $k$-inconsistent.
        \end{itemize}
        It has the \emph{tree property of the second kind} ($\TP_2$) if it has $k$-$\TP_2$ for some $1<k<\omega$. 
     \item[($7_n$)] $n$-\emph{strong order property} $(\SOP_n)$ for some $n\geq 3$, if $|x| = |y|$ and there are $(a_i)_{i<\omega}$ such that for all $i<j$, we have $\fhi(a_i;a_j)$ and the set 
    \[\{\fhi(x_1;x_2),\fhi(x_2;x_3),...,\fhi(x_{n-1},x_n),\fhi(x_n,x_1)\}\]
    is inconsistent.
    \end{enumerate}
    A complete first-order theory $T$ is 
    \begin{itemize}
        \item \emph{stable} if no formula has $\OP$ in $T$. It is said to be \emph{unstable} otherwise. 
        \item \emph{simple} if no formula has $\TP$ in $T$.
        \item $Np$ if no formula has $p$ in $T$, where $p \in \{\IP,\SOP,\TP_1,\TP_2, \SOP_n\}$; otherwise, we say that $T$ has $p$. For instance, $T$ is $\NIP$ if no formula has $\IP$ in $T$; if there is a formula with $\IP$ in $T$, we say that $T$ has $\IP$. 
    \end{itemize}
\end{defn}
The reader can find a diagram of known relations and implications between these properties below. 
\begin{fact}\label{implications-classical-div-lines}
    The following relations hold for properties of complete first-order theories. 
    \begin{itemize}
        \item Stable = $\NIP \cap \NSOP$.
        \item Simple = $\NTP_1 \cap \NTP_2$.
    \end{itemize}
\vspace{12pt}
\begin{center}
    \begin{tikzpicture}[node distance={12mm}]
    \tikzset{edge/.style = {->}}
        \node(OP){ $\OP$};
        \node(IP)[above of =OP]{$\IP$};
        \node(TP)[right of=OP]{$\TP$};
        \node(TP2)[above of =TP]{$\TP_2$};
        \node(TP1)[right of = TP]{$\TP_1$};
        \node(SOPn)[right of=TP1,, xshift=3mm]{$\SOP_n$};
        \node(SOPn+1)[right of =SOPn, xshift=5mm]{$\SOP_{n+1} $};
        \node(dots2)[right of=SOPn+1, xshift=3mm]{$...$};
        \node(SOP)[right of=dots2]{$\SOP$};
        \draw[edge] (TP) to (OP);
        \draw[edge] (IP) to (OP);
        \draw[edge] (TP2) to (IP);
        \draw[edge] (TP2) to (TP);
        \draw[edge] (TP1) to (TP);
        \draw[edge] (SOPn) to (TP1);
        \draw[edge] (SOPn+1) to (SOPn);
        \draw[edge] (dots2) to (SOPn+1);
        \draw[edge] (SOP) to (dots2);
        
    \end{tikzpicture}
\end{center}
\end{fact}
\begin{proof}
    For the proofs of the above we refer to \cite{casanovas-NIP}, \cite{conant-dividinglines}, \cite{TP1-kimkim} and \cite{shelah1990classification}. 
\end{proof}

\subsection{Generalized indiscernibles}
We review the basic definitions and notions around generalized indiscernibles, introduced by  Shelah in \cite{shelah1990classification} and developed by Scow in \cite{Scowthesis,Scow2015_Ind-emtypes-RamClass}.

\begin{defn}
    Fix a language $\L'$, $\L'$-structures $I$ and $J$ and tuples $b_i = f(i)$, $c_j = g(j)$ of some fixed length $\ell$ for some $f\colon I \to \U^\ell$ and $g\colon J \to \U^\ell$. 
    \begin{enumerate}
        \item We say that $(b_i)_{i \in I}$ is an \emph{$I$-indexed generalized indiscernible} (or just \emph{$I$-indiscernible}) if for all $n<\omega$ and for all $\Bar{i} = i_0,...,i_{n-1}$ and $\Bar{j} = j_0,...,j_{n-1}$ from $I$, we have 
        \[\mathrm{qftp}_{\L'}(\Bar{i}) = \mathrm{qftp}_{\L'}(\Bar{j}) \Rightarrow \tp_{\L}(b_{\Bar{i}}) = \tp_{\L}(b_{\Bar{j}})\]

        \item We say that $(b_i)_{i \in I}$ is \emph{based on} $(c_j)_{j \in J}$ if for any finite set of formulas $\Sigma$, for any $n<\omega$ and any tuple $\Bar{i} = i_0...i_{n-1}$ from $I$  there is a tuple $\Bar{j} = j_0...j_{n-1}$ from $J$ with $\mathrm{qftp}_{\L'}(\Bar{i}) = \mathrm{qftp}_{\L'}(\Bar{j})$, such that 
        \[\tp_\Sigma(b_{\Bar{i}}) = \tp_\Sigma(c_{\Bar{j}})\]

        \item We say that $I$ \emph{has the modeling property} if given any $(c_i)_{i \in I}$ there exists an $I$-indiscernible $(b_i)_{i \in I}$ based on $(c_i)_{i \in I}$. 
    \end{enumerate}
\end{defn}
\begin{fact}[Stretching indiscernibles, see \cite{Scowthesis}]\label{stretching-indiscernibles}
Let $I$ be an $\L'$-structure, $(b_i)_{i \in I}$ $I$-indiscernible. Then, for any $\L'$-structure $J$ with $\age(J) \subseteq \age(I)$, we can find $J$-indiscernible $(c_j)_{j \in J}$ based on $(b_i)_{i \in I}$.
\end{fact}
In her thesis (\cite{Scowthesis}), Scow showed a profound connection between Ramsey classes (See, e.g., \cite{Scowthesis}, Definition 3.2.6) and the modeling property for generalized indiscernibles. 
\begin{fact}\label{ramsey-modelingproperty}
    Let $I$ be a structure in a finite relational language and $\K = \age(I)$. Then $\K$ is a Ramsey class if and only if $I$ has the modeling property.
\end{fact}

\begin{fact}\label{gener-ind-hypergraph}
    Let $k\geq 2$. Applying the ``main theorem" in \cite{NESETRIL1983183}, we have that the class of finite ordered $k$-hypergraphs is a Ramsey class.
    By Fact \ref{ramsey-modelingproperty}, the generic ordered $k$-hypergraph has the modeling property.
\end{fact}

\section{Patterns of consistency and inconsistency}\label{Sec-patterns}
In this section, we want to describe a general framework for talking about combinatorial properties defining dividing lines. We start with a visual representation of some of the properties of formulas defined in Definition \ref{classical-div-lines}. 

The order property $(\OP)$ for a formula $\fhi(x;y)$ can be represented as follows: there are $(b_i)_{i<\omega}$ from $ \U^y$ such that, for each $i<\omega$ we have
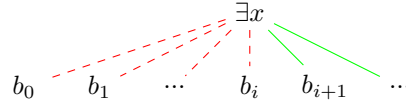
\begin{center}
\begin{tikzpicture}
    \node(b0){$b_0$};
    \node(b1)[right of = b0]{$b_1$};
    \node(dots1)[right of =b1]{...};
    \node(bi)[right of = dots1]{$b_i$};
    \node(bi+1)[right of=bi]{$b_{i+1}$};
    \node (dots2)[right of =bi+1]{...};
    \node(witness)[above of = bi]{$\exists x$};
    \draw[dashed] (witness) -- (b0);
    \draw[dashed] (witness) -- (b1);
    \draw[dashed] (witness) -- (dots1);
    \draw[dashed] (witness) -- (bi);
    \draw[] (witness) -- (bi+1);
    \draw[] (witness) -- (dots2);
\end{tikzpicture}
\end{center}
where the dashed lines represent $\neg \fhi$ and the solid lines represent $\fhi$. Similarly, the independence property $(\IP)$ can be visualized as follows: there are $(b_i)_{i<\omega}$ from $\U^y$ such that for every $X\subseteq \omega$ (circled elements below):

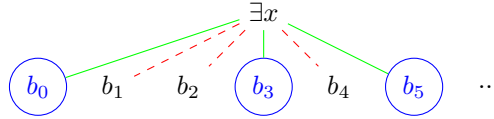
\begin{center}
\begin{tikzpicture}
    \node(b0)[circle,draw]{$b_0$};
    \node(b1)[right of = b0]{$b_1$};
    \node(b2)[right of =b1]{$b_2$};
    \node(b3)[circle,draw][right of =b2]{$b_3$};
    \node(b4)[right of = b3]{$b_4$};
    \node(b5)[circle,draw][right of=b4]{$b_5$};
    \node (dots2)[right of =b5]{...};
    \node(witness)[above of = b3]{$\exists x$};
    \draw[] (witness) -- (b0);
    \draw[dashed] (witness) -- (b1);
    \draw[dashed] (witness) -- (b2);
    \draw[dashed] (witness) -- (b4);
    \draw[] (witness) -- (b3);
    \draw[] (witness) -- (b5);
\end{tikzpicture}
\end{center}

A new element of complexity is added for describing the strict-order property $(\SOP)$: there is $(b_i)_{i<\omega}$ from $\U^y$ such that for all $i<\omega$ we have

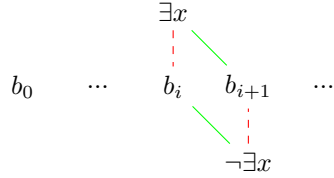
\begin{center}
\begin{tikzpicture}
    \node(b0){$b_0$};
    \node(dots1)[right of =b0]{...};
    \node(bi)[right of = dots1]{$b_i$};
    \node(bi+1)[right of=bi]{$b_{i+1}$};
    \node (dots2)[right of =bi+1]{...};
    \node(witness)[above of = bi]{$\exists x$};
    \node(nowitness)[below of = bi+1]{$\neg \exists x$};
   
    \draw[dashed] (witness) -- (bi);
    \draw[] (witness) -- (bi+1);
     \draw[dashed] (nowitness) -- (bi+1);
    \draw[] (nowitness) -- (bi);
\end{tikzpicture}
\end{center}
We see that the these properties are realized by satisfying consistency ($\exists x$) or inconsistency ($\neg \exists x$) requests on positive (solid line i.e.\ $\fhi$) or negative (dashed line i.e.\ $\neg \fhi$) instances of the formula $\fhi(x;y)$ on subsets of the set of parameters. 
We are now ready to define \emph{patterns} as pairs $(\C,\I)$ where $\C$ and $\I$ will be sets of conditions of the form $(X^+,X^-)$; the conceptual correspondence should be 
\begin{align*}
    &\C \longleftrightarrow \exists x  &X^+ \longleftrightarrow \fhi\\
    &\I \longleftrightarrow \neg \exists x
    &X^- \longleftrightarrow \neg \fhi
\end{align*}
\begin{defn}\label{def-patterns}
    Fix $n<\omega$. An $n$-\emph{pattern (of consistency and inconsistency)} is a pair $(\C,\I)$ where $\C,\I \subseteq \mathcal{P}(n) \times \mathcal{P}(n) \setminus \{(\emptyset,\emptyset)\}$.
    
     A pair $(X^+,X^-)$ in $\C$ (resp. in $\I$) is called a \emph{consistency} (resp. \emph{inconsistency}) \emph{condition}, or just \emph{condition}, of the pattern $(\C,\I)$. 
\end{defn}
\begin{defn}\label{defn-exhibition-of-patterns}
    Let $\fhi(x;y)$ be a partitioned formula, let $n<\omega$ and let $(\C,\I)$ be a $n$-pattern. We say that $\fhi(x;y)$ \emph{exhibits} (or \emph{realizes}) the pattern $(\C,\I)$ (in $T$) if there is $B = (b_i)_{i<n}$ from $\U^y$ such that 
    \begin{itemize}
        \item[(i)] For all $\A = (A^+,A^-) \in \C$ the partial $\fhi$-type 
        \[p_\A^B(x) = \{\fhi(x;b_i) \: i\in A^+\}\cup\{\neg \fhi(x;b_j) \: j \in A^-\}\] is consistent, and
        \item[(ii)] for all $\mathcal{Z} = (Z^+,Z^-) \in \I$, the partial $\fhi$-type
        \[p_\mathcal{Z}^B(x) = \{\fhi(x;b_i) \: i \in Z^+\}\cup\{\neg \fhi(x;b_j) \: j \in Z^-\}\] is inconsistent.
    \end{itemize}
\end{defn} 

In the next proposition we show how to characterize some of the dividing lines using exhibition of patterns.
\begin{prop}\label{div-lines-as-patterns} Let $\fhi(x;y)$ be a partitioned formula. Then 
    \begin{enumerate}
    
    \item $\fhi(x;y)$ has $\OP$ if and only if, for all $n<\omega$, the formula $\fhi(x;y)$ exhibits the $n$-pattern $(\C_n,\I_n)$ where
    \[\C_n = \{(\{i,i+1,...,n-1\},\{0,...,i-1\}) \: i<n\}\]
    \[\I_n = \emptyset\]
    \item  $\fhi(x;y)$ has $\IP$ if and only if, for all $n<\omega$, the formula $\fhi(x;y)$ exhibits the $n$-pattern $(\C_n,\I_n)$ where
    \[\C_n = \{(X,n\setminus X) \: X \subseteq n\}\]
    \[\I_n = \emptyset\]
    \item  $\fhi(x;y)$ has $\SOP$ if and only if, for all $n<\omega$, the formula $\fhi(x;y)$ exhibits the $n$-pattern $(\C_n,\I_n)$ where
    \[\C_n = \{(\{i+1\},\{i\}) \: i<n-1\}\]
    \[\I_n = \{(\{i\},\{i+1\}) \: i<n-1\}\]
    
    \item For $1<k<\omega$, $\fhi(x;y)$ has $k$-$\TP$ if and only if, for all $n<\omega$, the formula $\fhi(x;y)$ exhibits the $n^{<n}$-pattern $(\C_n,\I_n)$ where
    \[\C_n = \{(\text{ path}_i,\emptyset) \: i \in n^n\}\]
    \[\I_n = \{(\{i_0,...,i_{k-1}\},\emptyset) \: i_0,...,i_{k-1} \text{ children of the same node}\}\]
    where $\{path_i \: i\in n^n\}$ is the collection of all the paths in the tree $n^{<n}$.
    
    \item  $\fhi(x;y)$ has $\TP_1$ if and only if, for all $n<\omega$, the formula $\fhi(x;y)$ exhibits the $n^{<n}$-pattern $(\C_n,\I_n)$ where
    \[\C_n = \{(\text{ path}_i,\emptyset) \: i \in n^n\}\]
    \[\I_n = \{(\{i,j\},\emptyset) \: i,j \text{ are incomparable elements in the tree } n^{<n}\}\]
    with a similar notation as above.
    \item For $1<k<\omega$, $\fhi(x;y)$ has $k$-$\TP_2$ if and only if, for all $n<\omega$, the formula $\fhi(x;y)$ exhibits the $n^2$-pattern $(\C_n,\I_n)$ where
    \[\C_n = \{(\text{ path}_i,\emptyset) \: i \in n^n\}\]
    \[\I_n = \{(\{i_0,...,i_{k-1}\},\emptyset) \: i_0,...,i_{k-1} \text{ are on the same row }\}\]
    Here, by a path in $n^2$, we mean a function $f\colon n \to n $ i.e.\ a choice of one element per row in the array $n^2$. 
    \end{enumerate}
\end{prop}
\begin{proof}
    The characterization is clear. We get the usual definitions of the properties using compactness.  
\end{proof}

The above proposition shows that exhibition of patterns gives a general framework for talking about dividing lines in first-order theories. There are some exceptions. For instance, it is not clear whether the $\SOP_n$ hierarchy, for $n\geq 4$, can be characterized using pattern exhibition. There exists a characterization of $\SOP_3$ (\cite{kaplan2023generic}, Lemma 7.3) that can be described through patterns.
We also want to point out that, in the above proposition, $\SOP$ is the only property that is described using both consistency and inconsistency condition on both positive and negative instances of the formula. 
\begin{defn}
We say that an $n$-pattern $(\C,\I)$ is \emph{exhibitable} if there is a theory $T$ and a formula $\fhi(x;y)$ such that $\fhi$ exhibits $(\C,\I)$ in $T$. 
\end{defn}
There are ``syntactical reasons" that make a pattern not exhibitable; with the next definition we want to exclude those reasons from the picture. 
\begin{defn}\label{Reasonable}
    Fix $n<\omega$. A \emph{reasonable $n$-pattern (of consistency and inconsistency)} is an $n$-pattern such that 
    \begin{itemize}
        \item[(i)] For all $(Z^+,Z^-) \in \I$ and for all $(Y^+,Y^-) \in \C$, we have that $Z^\epsilon \not \subseteq Y^\epsilon$ for some $\epsilon \in \{+,-\}$. 
        \item[(ii)] For all $(A^+,A^-) \in \C$, $A^+ \cap A^- = \emptyset$.
        \item[(iii)] For all $(Z^+,Z^-) \in \I$, $Z^+ \cap Z^- = \emptyset$. 
    \end{itemize}
\end{defn}
\begin{rem}\label{rem_reas_exhib}
    In the above definition, a failure of condition (iii) doesn't really make a pattern not exhibitable; however, an inconsistency condition $(Z^+,Z^-)$ for which $Z^+ \cap Z^- \neq \emptyset$ doesn't add any meaningful request to the pattern and thus we avoid it.
    We also want to point out that ``reasonability" does not imply ``exhibitability". For example, a pattern in which $\I$ contains both $(\{0\},\emptyset)$ and $(\emptyset,\{0\})$ is not exhibitable. 
\end{rem}
Even assuming reasonability, deciding whether a pattern is exhibitable or not is a complex problem. More precisely, let $\mathrm{PSAT}$ be the following decision problem: given $n<\omega$ and a reasonable $n$-pattern $(\C,\I)$, is $(\C,\I)$ exhibitable?
\begin{prop}
    The decision problem $\mathrm{PSAT}$ is $\mathrm{NP}$-complete.
\end{prop}
The idea of the proof is to first provide a more concrete characterization of exhibitability and then show how to reduce $\mathrm{PSAT}$ to $\mathrm{SAT}$ (\emph{the Boolean satisfiability probelm}) and vice-versa. The complete proof of the above will appear in the author's PhD thesis. 

\begin{defn}
    Fix $n<\omega$. A $n$-pattern $(\C,\I)$ is called \emph{condition complete}, or just \emph{complete}, if each one of its conditions is of the form $(X,n\setminus X)$.
\end{defn}
\begin{rem}
    Every complete pattern $(\C,\I)$ with $\C \cap \I = \emptyset$ is reasonable.  
\end{rem}
\begin{rem}
    As a justification for the above nomenclature, consider that, given any $n$-pattern $(\C,\I)$ and a formula $\fhi(x;y)$ exhibiting it with witnesses $B = (b_i)_{i<n}$, each condition $(X^+,X^-)$ is associated with a $\fhi$-type over $B$:
    \[p_{X^+,X^-}^B(x) = \{\fhi(x;b_i) \: i \in X^+\} \cup \{\neg \fhi(x;b_j) \: j \in X^-\}\]
    that is asked to be consistent or inconsistent based on where the condition lies (in $\C$ or in $\I$ respectively).
    Saying that a pattern is complete is saying that those types corresponding to conditions are all complete $\fhi$-types over $B$. 
\end{rem}
\begin{defn}
     A $n$-pattern $(\C,\I)$ is called \emph{fully complete} if it is complete, $\C \neq \emptyset$ and every condition of the form $(X, n\setminus X)$ is either in $\C$ or in $\I$ but not both.
\end{defn}

\begin{prop}\label{Fully-complete-patt_are_exhibitable}
    Every fully complete pattern is exhibitable.
\end{prop}
\begin{proof}
    We actually show something stronger: there is a theory $T$ and a formula $\fhi$ such that, for every $n<\omega$ and fully complete $n$-pattern, $\fhi$ exhibits it. 

    Let $T$ be the theory of infinite atomic Boolean algebras in the language $\L = \{\cup,\cap, (-)^c,0,1\}$, work in a model $\B$, let $\psi(x;y) = \mathrm{At}(x) \land x \leq y$ and define
    \[\fhi(x;yw_0w_1) = (\psi(x;y) \land w_0=w_1)\lor(\neg \psi(x;y) \land w_0 \neq w_1)\]
    where $\mathrm{At}(x)$ is the formula saying that $x$ is an atom. Fix $n<\omega$ and a fully complete $n$-pattern $(\C,\I)$. Enumerate $\C = \{(A_j,n\setminus A_j) \: j<k\}$. Since the pattern is fully complete, $\C \neq \emptyset$, so $k \geq 1$. Moreover, let $a_0,...,a_{k-1}$ be $k$ distinct atoms of $\B$.
    
    First, assume that $(\emptyset,n) \notin \I$. Let $c$ any element of $\B$ and, for each $i<n$, define 
    \[b_i = \bigcup\{a_j \: j<k, i \in A_j\}\]
    and $c_i^0 = c_i^1 = c$.
    We claim that $(b_ic_i^0c_i^1)_{i<n}$ witness the exhibition of $(\C,\I)$ by $\fhi$. Indeed, let $j<k$ and consider the condition $(A_j,n \setminus A_j) \in \C$. We have that $i \in A_j$ if and only if $a_j \leq b_i$ if and only if $a_j \models \mathrm{At}(x) \land x \leq b_i$ if and only if $a_j \models \fhi(x;b_ic_i^0c_i^1)$.
    This holds for every $j<k$ and thus any consistency condition is satisfied. Let $(Z,n\setminus Z) \in \I$. By assumption $Z \neq \emptyset$. Assume, toward contradiction, that there is $a \in \U^x$, such that $\models \fhi(a;b_ic_i^0c_i^1)$ if and only if $i \in Z$. Since $Z \neq \emptyset$, there is at least one $i$ such that $a$ is an atom below $b_i$. Thus there exists $j<k$ such that $a = a_j$. But now $i \in Z$ if and only if $\models \fhi(a_j;b_ic_i^0c_i^1)$ if and only if $a_j \leq b_i$ if and only if $i \in A_j$. Thus $Z = A_j$ and $(Z,n\setminus Z) \in \C$. Contradiction. Thus, in the case that $(\emptyset,n) \notin \I$, we are done. 
    
    Assume now that $(\emptyset,n) \in \I$. Fix two distinct elements of the model $c,c'$ and for $i<n$, $i \notin A_0$, define 
    \[b_i = \bigcup\{a_j \: j<k, i \in A_j\} \quad \text{ and take } c_i^0 = c = c_i^1 \]
    For $i \in A_0$, define 
    \[b_i = \bigcup\{a_j \: j<k, i \notin A_j\}  \quad \text{ and take } c_i^0 = c \neq c' = c_i^1 \]
    Let $B = (b_ic_i^0c_i^1)_{i<n}$.
    Fix $(A_j,n\setminus A_j) \in \C$ and suppose $j \neq 0$. 
    If $i \in n \setminus A_0$, we have 
    \[a_j \models \fhi(x;b_i\Bar{c_i}) \iff \mathrm{At}(a_j) \land (a_j \leq b_i) \iff a_j \leq b_i \iff i \in A_j \setminus A_0\]
    If $i \in A_0$, we have 
    \[a_j \models \fhi(x;b_i\Bar{c_i}) \iff \neg \mathrm{At}(a_j) \lor (a_j \not \leq b_i)\]
    \[\iff a_j \not \leq b_i \iff a_j \leq b_i^c = (\bigcup_{i \notin A_j} a_j)^c = \bigcap_{i \notin A_j}a_j^c \iff i \in A_j\]
    Hence, if $(A_j, n \setminus A_j) \neq (A_0,n\setminus A_0)$, $a_j \models p_{A_j,n \setminus A_j}^B(x)$. 
    Consider now $(A_0,n \setminus A_0)$ and fix an element $b$ that is not an atom. Then, if $i \in A_0$, since $b \models \neg \mathrm{At}(x)$, $b \models \fhi(x;b_icc')$. 
    If $i \notin A_0$, then, since $b \not \models \mathrm{At}(x) $, $b \models \neg\fhi(x;b_icc)$. 
    Thus, we showed that each consistency condition is satisfied. 

    Fix now $(Z,n\setminus Z) \in \I$. First assume that $Z = \emptyset$; we want to show that for any $a \in \mathcal{B}$, there is $i<n$ such that $a \models \fhi(x;b_i\Bar{c}_i)$. If $a \notin \mathrm{At}(\B)$, then $a \models \neg \mathrm{At}(x)$ and hence $a \models \fhi(x;b_i,cc')$ for any $i \in A_0$. If $a \in \mathrm{At}(\B) \setminus (a_j)_{0<j<k}$, then $a \not \leq b_i$ for $i \in A_0$ and hence $a \models \fhi(x;b_icc')$. If $a = a_j \neq a_0$, then $a_j \models \fhi(x;b_icc)$ for $i \in A_j$ (Note that $A_j \neq \emptyset$ since $(\emptyset,n) \in \I$). Hence $p_{(\emptyset,n)}^B(x)$ is inconsistent.

    Assume $Z \neq \emptyset$ and suppose that there is $a \in \mathcal{B}$ such that 
    \[a \models \fhi(x;b_i\Bar{c}_i) \iff i \in Z\]
    If $a$ is not one of the $(a_j)_{j<k}$, then $a \models \fhi(x;b_i\Bar{c}_i) \iff i \in A_0$ and hence $A_0 = Z$. Contradiction. 
    
    If $a = a_j$ for some $0<j<k$, then $a=a_j \models \fhi(x;b_i\Bar{c}_i) \iff i \in A_j$ and hence $A_j = Z$. Contradiction.
    If $a = a_0$, then if $i \in A_0$ then $a_0 = a \not \leq b_i$ and hence $a \models\fhi(x;b_icc')$; on the other hand, if $i \notin A_0$, then $a_0 = a \not \leq b_i$ and hence $a\not\models \fhi(x;b_icc) $. Hence $Z = A_0$. Contradiction.
    Hence $p_{(Z,n\setminus Z)}^B(x)$ is inconsistent.

    In conclusion, for every $n<\omega$, $\fhi$ exhibits every fully complete $n$-pattern. 
\end{proof}
\begin{rem}
    In the above proof we gave an example of a formula $\fhi$ that, in a theory $T$, exhibits, for every $n<\omega$, every fully complete $n$-pattern. In the next proposition, we will show that any exhibitable pattern can be extended (in some sense) to a fully complete pattern. Thus, the above proof gives an example of a theory $T$ and a formula $\fhi$ such that for any $n<\omega$ and any exhibitable $n$-pattern, $\fhi$ exhibits it. This property is what, in the next section, we will call \emph{Straight up Maximality} (similar to Shelah's nomenclature from \cite{Shel-modeltheory}); thus the theory of infinite atomic Boolean algebras is an example of a straight up maximal theory.  
\end{rem}

\begin{prop}\label{fcp_ext}
    Fix $n<\omega$ and let $(\C,\I)$ any exhibitable $n$-pattern. Then, there exists a fully complete $n$-pattern $(\C',\I')$ such that, if a formula exhibits $(\C',\I')$, then it exhibits $(\C,\I)$. 
\end{prop}
\begin{proof}
    Since $(\C,\I)$ is exhibitable, there is a theory $T$, a formula $\fhi(x;y)$ and parameters $B = (b_i)_{i<n}$ such that $\fhi$ exhibits $(\C,\I)$ in $T$ with $B$. As a convention for this proof, a complete $\fhi$-type over $B$ is not necessarily consistent.
    Let $\mathrm{Con}(\fhi,B)$ be the set of complete consistent $\fhi$-types over $B$ and, similarly, let $\mathrm{Inc}(\fhi,B)$ be the set of complete inconsistent $\fhi$-types over $B$. To each complete $\fhi$-type $p$, we can associate a condition $(X_p,n\setminus X_p)$ where $X_p = \{i < n \: \fhi(x;b_i) \in p\}$. Define 
    \[\C' = \{(A_p,n \setminus A_p) \: p \in \mathrm{Con}(\fhi,B)\}\]
    \[\I' = \{(Z_p,n \setminus Z_p) \: p \in \mathrm{Inc}(\fhi,B)\}\]
    Suppose $\psi(x;y)$ exhibits $(\C',\I')$ in $T$ with $B' = (b_i')_{i<n}$. 

    Given $(A^+,A^-) \in \C$, since $\fhi(x;y)$ exhibits $(\C,\I)$ with $B$, we have that $p_{(A^+,A^-)}^{B,\fhi}(x)$ is consistent and thus it is contained in some $p \in \mathrm{Con}(\fhi,B)$. Thus, there is $(A,n\setminus A) \in \C'$ such that $(A^+,A^-) \subseteq (A,n \setminus A)$. Since $p_{(A,n\setminus A)}^{B',\psi}(x)$ is consistent, it must be that $p_{(A^+,A^-)}^{B',\psi}(x)$ is consistent too. 

    Let $(Z^+,Z^-) \in \I$ and assume, toward contradiction, that $p_{(Z^+,Z^-)}^{B',\psi}(x)$ is realized by some $a \in \U^x$. By the choice of $\fhi$ and $B$, we have that $p_{(Z^+,Z^-)}^{B,\fhi}(x)$ is inconsistent; thus, for any $X \subseteq n$ such that $(Z^+,Z^-) \subseteq (X,n \setminus X)$, we have that $(X,n \setminus X) \in \I'$. But now $\tp^\psi(a/B')$ gives us a set $X \subseteq n$ such that $(Z^+,Z^-) \subseteq (X,n\setminus X)$ which is in $\I'$ and the corresponding type has a realization. Contradiction.
\end{proof}

\section{Notions of Maximality}\label{Sec-Maximality}
In this section we define a notion of maximal complexity for first-order theories by requesting the exhibition (by a single formula) of all the exhibitable patterns. 
\begin{defn}
    A formula $\fhi(x;y)$ is \emph{Straight up Maximal} (or just $\SM$) if, for every $n<\omega$ and for every exhibitable $n$-pattern $(\C,\I)$, $\fhi(x;y)$ exhibits $(\C,\I)$.
    A theory $T$ is $\SM$ if there is a formula which is $\SM$; otherwise we say that $T$ is $\NSM$. 
\end{defn}
Other notions of maximality were previously defined by Shelah (\cite{Shel-modeltheory}) and Cooper (\cite{cooper}). We will see how both of them agree with our notion of maximality and we will give some examples of theories which present/omit this property.  

\subsection{Shelah's maximality}
In \cite{Shel-modeltheory}, Shelah defined ``an extreme condition of the form of unstability" called Straight Maximality.
\begin{defn}
    A formula $\fhi(x;y)$ is \emph{Shelah-$\SM$} (Shelah's Straight maximal) if for every $n<\omega$ and for every non-empty $\mathscr{F} \subseteq 2^n$ there are $b_0,...,b_{n-1} \in \U^y$ such that: Given $f \in 2^n$, $f \in \mathscr{F}$ if and only if there is $a \in \U^x$ such that for every $i<n$, $\models \fhi(a;b_i) \iff f(i) =1$.
\end{defn}
We can characterize the above property using patterns.
\begin{prop}\label{fcp-shelahSM}
    A formula $\fhi(x;y)$ is Shelah-$\SM$ if and only if $\fhi(x;y)$ exhibits every fully complete pattern. 
\end{prop}
\begin{proof}
    Assume $\fhi$ is Shelah-$\SM$, fix $n<\omega$ and let $(\C,\I)$ a fully complete $n$-pattern. Define $\mathscr{F} = \{\mathrm{char}(X) \: (X,n\setminus X) \in \C\} \subseteq 2^n$ where $\mathrm{char}(X)\colon n \to 2$ is the characteristic function of $X \subseteq n$. Since $\C \neq \emptyset$ (given that the pattern is fully complete), we have that $\mathscr{F}$ is not empty. Let $B = \{b_0,...,b_{n-1}\}$ given by Shelah-$\SM$. Now, fix a condition $(Y^+,Y^-) = (Y,n\setminus Y) \in \C$. This corresponds to a function $f\colon n \to 2$, $f \in \mathscr{F}$, such that $f(i) = 1$ if and only if $i \in Y$. Hence there is $a \in \U^x$ such that $\models \fhi(a;b_i)$ if and only if $f(i) = 1$ if and only if $i \in X$; thus $a$ witness the consistency of $p_{Y^+,Y^-}^B(x)$. Now, given a condition of inconsistency $(Z^+,Z^-) = (Z,n\setminus Z) \in \I$, if there was $a \in \U^x$ such that $\models \fhi(a;b_i)$ if and only if $i \in Z$ if and only if $\mathrm{char}(Z)(i) = 1$, we would have that $\mathrm{char}(Z) \in \mathscr{F}$; but it's not. Contradiction. Hence the type $p_{Z^+,Z^-}^B(x)$ is inconsistent. This shows that $\fhi$ exhibits $(\C,\I)$. In conclusion (of the proof of this direction), $\fhi$ exhibits every fully complete pattern.

    On the other hand, assume that $\fhi$ exhibits every fully complete pattern. Given $n<\omega$ and $\emptyset \neq \mathscr{F} \subseteq 2^n$, define the $n
    $-pattern. $(\C,\I)$ as follows
    \[\C = \{(f^{-1}(\{1\}), f^{-1}(\{0\})) \: f \in \mathscr{F}\}\]
    \[\I = \{(f^{-1}(\{1\}), f^{-1}(\{0\})) \: f \in 2^n \setminus \mathscr{F}\}\]
    Clearly this is a fully complete pattern.
    Along the lines of the proof of the previous part, it's clear that a witness $B = \{b_0,...,b_{n-1}\}$ of the exhibition of $(\C,\I)$ by $\fhi$ is also a witness of $\fhi$ being Shelah-$\SM$. 
\end{proof}
\begin{prop}
    A formula is Shelah-$\SM$ if and only if it is $\SM$.
\end{prop}
\begin{proof}
Let $\fhi(x;y)$ be a $\SM$ formula i.e.\ it realizes every exhibitable pattern. Since every fully complete pattern is exhibitable, $\fhi$ exhibits every fully complete pattern. Hence $\fhi$ is Shelah-$\SM$ by Proposition $\ref{fcp-shelahSM}$. 

On the other hand, assume that $\fhi(x;y)$ is Shelah-$\SM$; then it realizes every fully complete pattern. Given an exhibitable pattern $(\C,\I)$, by Proposition \ref{fcp_ext}, there is a fully complete pattern $(\C',\I')$ such that, if a formula realizes $(\C',\I')$ then it realizes $(\C,\I)$. Since $\fhi$ realizes every fully complete pattern, it realizes every exhibitable pattern. In conclusion, it is $\SM$. 
\end{proof}

\subsection{Cooper's maximality}
In \cite{cooper}, Cooper defined a condition called $1$, requesting the existence of a formula which defines, in some sense, all the possible Venn diagrams. They then proved that this property implies every other property of this kind. For more details about the background and the maximality property, we refer to \cite{cooper}. 

\begin{defn}[Cooper, \cite{cooper}] \label{prop-of-formulas}
\begin{itemize}
    \item A \emph{property of formulas} is any sequence $\rho = (\rho_i)_{i<\omega}$ of consistent quantifier free (q.f.) formulas of $\mathrm{BA}$ (the theory of Boolean Algebras with the language $\L = \{0,1,(-)^c,\cup,\cap\}$).
\item If $\fhi(x,y)$ is a formula of a complete theory $T$ and $\psi(z)$ is a q.f.\ formula of $\mathrm{BA}$, we say that $\fhi(x,y)$ \emph{admits} $\psi(z)$ \emph{in} $T$ if 
\[\mathcal{P}(\M^x) \models \psi(\fhi(\M,a_0),...,\fhi(\M,a_{|z|-1}))\]
for some $\M \models T $ and some $a_0,...,a_{|z|-1} \in \M^x$ i.e. there is some model of $T$ and parameters such that, in the power set of the set of tuples of the model (of the right length) viewed as a Boolean algebra, the $\fhi$-definable sets given by the parameters satisfy $\psi$. 
Otherwise we say that $\fhi(x,y)$ \emph{omits} $\psi(z)$ in $T$.

\item If $\rho$ is a property of formulas, we say that a formula $\fhi(x,y)$ of a complete theory $T$ \emph{admits} $\rho$ \emph{in} $T$ if there exists a strictly increasing sequence $\alpha \in \omega^\omega$ such that $\fhi(x,y)$ admits $\rho_{\alpha(i)}$ in $T$ for every $i<\omega$.
Otherwise we say that $\fhi(x,y)$ \emph{omits} $\rho$ \emph{in} $T$.
Moreover we say that $T$ \emph{admits $\rho$} if there exists $\fhi(x,y)$ of $T$ such that $\fhi(x,y)$ admits $\rho$ in $T$. Otherwise $T$ \emph{omits $\rho$}.
\end{itemize}
\end{defn}

Cooper's notion of maximality (property $1$) is a property of formulas (as in Definition \ref{prop-of-formulas}) which implies all the other properties of formulas. In order to define $1$, Cooper started with an enumeration of all the consistent q.f.\ formulas of $\mathrm{BA}$, say $\{\psi_i \: i <\omega\}$, and defined
\[\rho_i = \bigwedge_{j<i} \psi_i\]
Finally, $1$ is defined as the property of formulas $(\rho_i)_{i<\omega}$.

The next definition gives us a more explicit equivalent (at the level of theories) version of property $1$. 
\begin{defn}[Cooper, \cite{cooper}]
For each $0<n<\omega$ we define the property $1^{(n)}$ as follows:

A formula $\fhi(x,y)$ admits $1^{(n)}$ in some complete theory $T$ if for arbitrarily large $n\leq m <\omega$ there are $(b_i)_{i<m}$ in some model $\U^y$ such that
\begin{itemize}
    \item[(i)] for any $l\leq m$ and for any $X \in \mathcal{P}(m)$ with $|X| = l$ there is $b_X$ such that 
    \[\bigcup_{i \in X}\fhi(\U,b_i) = \fhi(\U,b_X)\]
    \item[(ii)] for any $Y \in \mathcal{P}(m)$
    \[\bigcap_{i \in Y}\fhi(\U,b_i) = \emptyset \text{ if and only if } |Y|>n\]
\end{itemize}
In other words, $\fhi(x,y)$ admits $1^{(n)}$ if and only if for arbitrarily large $n\leq m <\omega$ there are $m$ $\fhi$-definable sets such that the union of any $l\leq m $ is still $\fhi$-definable but the intersection of any $l\leq m$ of them is non-empty if and only if $l \leq n$.
\end{defn}
\begin{fact}[Cooper, \cite{cooper}]\label{1^n_prop}
If $T$ is a complete theory and $0<n<\omega$, the following hold:
\begin{itemize}
    \item[(1)] If some formula $\fhi(x,y)$ of $T$ admits $1^{(n)}$ in $T$ then some formula $\psi(x,z)$ of $T$ admits $1^{(1)}$ in $T$.
    \item[(2)] If some formula $\psi(x,z)$ of $T$ admits $1^{(1)}$ in $T$, then some formula $\chi(x,w)$ of $T$ admits $1$ in $T$. 
\end{itemize}
\end{fact}
\begin{rem}
    The above fact shows that, at the level of theories, admitting the property $1^{(n)}$ implies admitting the property $1^{(1)}$ which, in turn, implies admitting the property $1$. Moreover, since each $1^{(n)}$ can be described as a property of formulas and each property of formulas is implied by $1$, we have that, at the level of theories, each $1^{(n)}$ is a different characterization of $1$.
\end{rem}
\begin{defn}
    A complete theory $T$ is \emph{Cooper maximal} if there is a formula which admits $1$ (equivalently $1^{(n)}$) in $T$. 
\end{defn}
We are now ready to prove that Cooper's maximality coincides with straight up maximality.

\begin{prop}\label{Coopermax_IFF_SM}
    Let $T$ be a complete theory. $T$ is Cooper maximal if and only if $T$ is $\SM$. 
\end{prop}
\begin{proof}
    Assume $T$ is $\SM$ and let $\fhi(x;y)$ be the $\SM$ formula witnessing it. 

    \begin{claim}
        for each $n<\omega$ there are $(b_i)_{i<n}$ such that 
        \[\fhi(\U,b_i)\cap \fhi(\U,b_j) = \emptyset \text{ for all }i\neq j < n\]
        for each $X \subseteq n$, there is $b_X$ such that 
        \[\fhi(\U,b_X) = \bigcup_{i \in X}\fhi(\U,b_i) \]
        Equivalently, $\fhi(x;y)$ admits $1^{(1)}$. 
    \end{claim}
\begin{claimproof}
        Since $\fhi(x;y)$ is $\SM$, it realizes every fully complete $n$-patterns and these patterns can be described by pairs $(C,I)$ where $C,I \subseteq \mathcal{P}(n)$ with $C \cap I = \emptyset$. Let $n< \omega$ and consider the following $2^n$-pattern:
\[C = \{\uparrow \{i\} \: i<n\}\]
\[I = \mathcal{P}(2^n) \setminus C\]
where $\uparrow\{i\} = \{X \subseteq n \: i \in X\}$. 
Clearly $I \cap C = \emptyset$.
Let $(b_X)_{X \subseteq n}$ be the parameters given by $\SM$ with this pattern.
First notice that $\fhi(\U,b_{\{i\}}) \cap \fhi(\U,b_{\{j\}}) = \emptyset$. Indeed, suppose $a \models \fhi(x,b_{\{i\}}) \land \fhi(x,b_{\{j\}})$. Let $A = \{\{k\} \subseteq n \: a \models \fhi(x,b_{\{k\}}) \}$. Since $A \in I$ we have that the $\fhi$-type of $a$ over $(b_X)_{X \subseteq n}$ is inconsistent; contradiction.

Let $X \subseteq n$. We want to show that 
\[\fhi(\U,b_X) = \bigcup_{i \in X} \fhi(\U,b_{\{i\}})\]
Let $a \models \fhi(x,b_{\{i\}})$ for some $i \in X$. Since $\mathcal{P}(n)\setminus X \in I$, it is inconsistent to be in $\fhi(\U,b_{\{i\}})$ but not in $\fhi(\U,b_X)$; hence $a \in \fhi(\U,b_X)$.
If $a \models \fhi(x,b_X)$. Since $\{X\} \in I$, it is inconsistent to be in $\fhi(\U,b_X)$ but not in all the $\fhi(x,b_{\{i\}})$, hence there is $i \in X$ such that $a \in \fhi(\U,b_{\{i\}})$.
\end{claimproof}
    
The claim shows that, if $\fhi(x;y)$ is $\SM$, then it admits $1^{(1)}$. By Fact \ref{1^n_prop}, there is a formula admitting $1$ and thus $T$ is Cooper maximal.

Assume now that $T$ is Cooper maximal and let $\fhi(x;y)$ admitting $1$. Since $1$ is the strongest property of formulas, it is enough to prove that $\SM$ can be described as a property of formulas. Indeed $\SM$ is a property of formulas described by: \[\rho_m := \bigwedge_{n<m} \bigwedge_{\substack{(\C,\I) \\ \text{ exhibitable $n$-p}}}( \bigwedge_{(Y^+,Y^-) \in \C}(\bigcap_{i\in Y^+}x_i \cap \bigcap_{j\in Y^{-}} (n \setminus x_j)) \neq 0 )\]\[\land \bigwedge_{(Z^+,Z^-) \in \I} (\bigcap_{i\in Z^+}x_i \cap \bigcap_{j\in Z^{-}} (n \setminus x_j)) = 0 )\]
This shows that, if $T$ is Cooper maximal, then it is $\SM$.
\end{proof}
\begin{rem}
    We want to point out that the above equivalence is only at the level of theories. Indeed, while a formula admitting $1$ is $\SM$, if a formula is $\SM$, it admits $1^{(1)}$ and thus (by Fact \ref{1^n_prop}) there is a (possibly different) formula admitting $1$. 
\end{rem}
\subsection{Examples of $\SM$ theories}
\begin{eg}[Infinite atomic Boolean algebra]
    In Proposition \ref{Fully-complete-patt_are_exhibitable}, we showed that the theory $T$ of infinite atomic Boolean algebras is $\SM$ by showing a formula that exhibits every exhibitable pattern. We now prove again this result using Cooper's characterization of $\SM$ (See Fact \ref{1^n_prop} and Proposition \ref{Coopermax_IFF_SM}): It is enough to find a formula with $1^{(1)}$ i.e.\ a formula $\fhi(x;y)$ such that, for any $n<\omega$ there are $(b_i)_{i<n}$ such that 
\begin{itemize}
    \item $\fhi(\U;b_i) \cap \fhi(\U;b_j) = \emptyset$ for all $i<j<n$, and 
    \item for each $X \subseteq n$, there is $b_X$ such that 
    \[\fhi(\U;b_X) = \bigcup_{i \in X} \fhi(\U;b_i)\]
\end{itemize}
Work in some model $\B$.
Let $\mathrm{At}(x)$ be the formula saying that $x$ is an atom and consider the formula 
\[\fhi(x;y) = \mathrm{At}(x) \land x \leq y\]
Fix $n<\omega$ and choose $(b_i)_{i<n}$ such that $\mathrm{At}(b_i)$ and $b_i \neq b_j$ for $i<j<n$. Now $\fhi (x;b_i) \land \fhi(x;b_j)$ is inconsistent for $i<j<n$. Moreover, if $X \subseteq n$, let  $b_X := \bigcup_{i \in X} b_i$; then 
\[\fhi(\B;b_X) = \bigcup\{b_i : b_i \leq b_X \land i<n\} = \bigcup_{i \in X} \fhi(\B;b_i)\]
In conclusion $\fhi(x;y)$ has $1^{(1)}$ and thus $T$ is $\SM$. 

\end{eg}
\begin{noneg}[Atomless Boolean algebras]
\label{Atomless-not-SM}
Consider the theory $T$ of atomless Boolean algebras in the language $\L = \{\cup,\cap,\neg,0,1\}$.
In \cite{cooper}, Cooper proved that $T$ does not admit $1$ and thus, by Proposition \ref{Coopermax_IFF_SM}, it is not $\SM$. 

In the next section we will show that, in $T$, the formula $\fhi(x;y) = 0 \neq x \leq y$ exhibits every pattern $(\C,\I)$ in which, for every $(X^+,X^-) \in \C \cup \I$, $X^- = \emptyset$ (see Definition \ref{positive-patt-defn} and Remark \ref{reas+pos->exhib}). Thus, $\fhi(x;y)$ is an example of what, in the next section, we will call a \emph{positive maximal formula} ($\PM$ formula, see Definition \ref{PM-defn}).
Although being $\PM$, $\fhi(x;y)$ does not exhibit every exhibitable pattern. For example, consider the $3$-pattern $(\C,\I)$ defined by 
\[\C = \{(\{0\},\emptyset),(\{1\},\emptyset)\}\]
\[\I = \{(\{0,1\},\emptyset),(\{0\},\{2\}),(\{1\},\{2\}),(\{2\},\{0,1\})\}\]
This pattern describes three sets $A_0,A_1,A_2$ with the conditions that $A_0 \neq \emptyset \neq A_1$, $A_0 \cap A_1 = \emptyset$ and $A_2 = A_0 \cup A_1$. 
This pattern is clearly exhibitable (in any theory with models of size at least $2$, by the formula $x = y_0 \lor x = y_1$).
On the other hand, it is easy to see that the formula $\fhi(x;y) = 0\neq x \leq y$ does not exhibit it in $T$.
We point out that this does not prove that $T$ is not $\SM$; it just shows that $\fhi(x;y)$ is not $\SM$. 
\end{noneg}

\begin{eg}[Skolem arithmetic]
Skolem arithmetic is the first-order theory of the structure $(\mathbb N,\cdot)$. First of all, we observe that we can define the number $1$, the divisibility relation and the set of prime numbers: 
\begin{itemize}
    \item $\mathrm{One}(x) = \forall y (x \cdot y = y)$.
    \item $x|y = \exists z (y = x \cdot z)$
    \item $P(x) = \neg \mathrm{One}(x) \land \forall y (y|x \to (\mathrm{One}(y) \lor y=x))$
\end{itemize}
To show that this theory is $\SM$ we use Cooper's characterization of $\SM$ showing a formula $\fhi(x;y)$ with $1^{(1)}$.
Define $\fhi(x;y) = P(x) \land x|y$ and enumerate the prime numbers $P(\mathbb N) = \{p_i \: i<\omega\}$. Fix $n< \omega$ and for each $i<n$ let $b_i = p_i$. Now, for $i<j<n$, $\fhi(\mathbb N,b_i) \cap \fhi(\mathbb N,b_j) = \{p_i\} \cap \{p_j\} = \emptyset$. Moreover, given $X \subseteq n$, take $b_X = \prod_{i \in X}p_i$. We have 
\[\fhi(\mathbb N,b_X) = \{p_i \: i \in X\} = \bigcup_{i \in X}\{p_i\} = \bigcup_{i \in X}\fhi(\mathbb N,b_i)\]
In conclusion, Skolem arithmetic is $\SM$.

A similar argument shows that the first-order theory of $(\mathbb Z,+, \cdot,0,1)$ is $\SM$. 
\end{eg}
Since all the previous examples of $\SM$ theories are not $\omega$-categorical, we now construct an $\omega$-categorical $\SM$ theory. 

\begin{eg}[$\omega$-categorical $\SM$ theory] Consider the complete theory of the two-sorted structure $(2^\omega,\mathrm{CLOP}(2^\omega))$ in the language of Boolean algebras for the second sort and a binary relational symbol interpreted as the membership relation between a point of the Cantor space and one of its clopen sets.

In order to prove $\omega$-categoricity, we present this example as the complete theory of a Fra\"issé limit of a class of finite models of a universal theory $T_0$, coding the two sorts with unary predicates. 
Consider the language $\L = \{S,B,R\}\cup \L_{\mathrm{BA}}$ where $\L_{\mathrm{BA}} = \{\cap, \cup, (-)^c,0,1\}$ is the language of Boolean algebras, $S$ and $B$ are unary predicates and $R$ is a binary relation. 
For each $\L_{\mathrm{BA}}$ term $t(\Bar{y})$ define an $\L$-formula $\psi_t(x,\Bar{y})$ as follows:
\begin{itemize}
    \item If $t$ is a variable $y$ then  $\psi_t(x,y) = R(x,y)$.
    \item if $t$ is $1$ then $\psi_t(x)$ is $x=x$. 
    \item if $t_1(\Bar{y})$,$t_2(\Bar{y}')$ are terms, then $\psi_{t_1 \cap t_2}(x,\Bar{y}\Bar{y}') = \psi_{t_1}(x,\Bar{y}) \land \psi_{t_2}(x,\Bar{y}')$.
    \item If $t(\Bar{y}) = (t'(\Bar{y}))^c$ then $\psi_{t}(x,\Bar{y}) = \neg \psi_{t'}(x,\Bar{y})$. 
\end{itemize}
Since $0$ and $\cup$ are definable from the other symbols, we can assume that the term $t$ does not contain them. 

Notice that, given an atomic $\psi \in \L_{\mathrm{BA}}$, $\psi$ is of the form $t_1(\Bar{y}) = t_2(\Bar{y}')$ for some $\L_{\mathrm{BA}}$-terms.

Consider the following $\L$-theory $T_0$: 
\begin{itemize}
    \item[($A_1$)\quad \ ] $S$ and $B$ partition the universe.
    \item[($A_2$)\quad\ ] $R \subseteq S \times B$.
    \item[($A_3$)\quad\ ] $B \models T_{BA}$ i.e.\ $B$ is a Boolean algebra.
    \item[($A_4^{t_1,t_2}$)] For $\L_{BA}$-terms $t_1(\bar{y})$ and $t_2(\bar{y}')$, we have \[ (\forall \Bar{y}\bar{y}' \in B)( t_1(\bar{y}) = t_2(\bar{y}') \rightarrow (\forall x \in S)(\psi_{t_1}(x,\bar{y})\leftrightarrow \psi_{t_2}(x,\bar{y}')))\]
\end{itemize}
We claim that the class $\K := Mod_{<\omega}(T_0)$ (i.e. the class of finite models of $T_0$) is a Fra\"issé class and the theory of its limit is $\SM$.
Intuitively, each structure in $\K$ can be described as follows: we have a set $S$, a Boolean algebra $B$ and a relation $R$ between the two disjoint subsets. The axiom scheme $(A_4^{t_1,t_2})$ makes sure that we have a homomorphism of Boolean algebras
\[h\colon B \to \mathcal{P}(S)\]
\[b \mapsto R(S,b)\]
Thus, we have a realization of a homomorphic image of $B$ as $R$-definable (over $B$) subsets of $S$. 

To show that $\K$ is a Fra\"issé class, we need to show that $\K$ has $\mathrm{HP}$, $\mathrm{JEP}$, $\mathrm{AP}$ and it is \emph{uniformly locally finite} i.e.\ there exists $f\colon \omega \to \omega$ that bounds the size of structures in $\K$ with respect to the number of their generators.
First of all, since $T_0$ is universal, we have that $\K$ is closed under substructures (i.e.\ $\K$ has $\mathrm{HP}$). Moreover, since the structure $(\emptyset,\{0,1\})$ is an initial object in $\K$ (i.e.\ it embeds uniquely in every structure in $\K$), we have that $\mathrm{JEP}$ is implied by $\mathrm{AP}$.   

Consider the function $f \colon \omega \to \omega$ defined as $f(n) = n+2^{2^n}$. Now, given $C \leq A \in \K$ with $C = \langle C_{gen}\rangle$, $C_{gen} = \{s_0,...,s_{l-1},b_0,...,b_{k-1}\}$ and $l+k=n$, it's easy to see that $|C| \leq l + 2^{2^k} \leq f(n)$. Thus $\K$ is uniformly locally finite. 
It remains to show that $\K$ has $\mathrm{AP}$.

\begin{lem}
    Given $A,A' \in \K$, there is an embedding $A \to A'$ if and only if there is an embedding of Boolean algebras $B(A) \to B(A')$ and a surjective homomorphism of Boolean algebras $\mathcal{P}(S(A')) \to \mathcal{P}(S(A))$ such that the following commutes
    \begin{center}
    \begin{tikzcd}
        B(A) \arrow[r, "h_A"] \arrow[d,"emb"]
         &\mathcal{P}(S(A)) \\
        B(A') \arrow[r, "h_{A'}"]
        & \mathcal{P}(S(A')) \arrow[two heads, u, "hom"]
\end{tikzcd}
\end{center}
where $h_A$ and $h_{A'}$ are the Boolean algebra homomorphisms given by $b \mapsto R(S,b)$. 
\end{lem}
\begin{proof}
    Suppose we have an embedding $A \to A'$. In particular we have a embedding of Boolean algebras $F \colon B(A) \to B(A')$ and an inclusion $S(A) \subseteq S(A')$ such that for every $s \in S(A)$, $R(s,b)$ if and only if $R(s,F(b))$. By taking $X \mapsto X \cap S(A)$, we have an surjective homomorphism of Boolean algebras $\mathcal{P}(S(A')) \twoheadrightarrow \mathcal{P}(S(A))$. Therefore, given $b \in B(A)$, we have
    \[h_{A'}(F(b)) \cap S(A) = R(S(A'),F(b)) \cap S(A) = R(S(A),b) = h_A(b)\]
    Conversely, given the above diagram, we have a Boolean algebra embedding $B(A) \to B(A')$ and by Stone duality, we get an injection $S(A) \to S(A')$. Commutativity of the diagram implies that the relation $R$ is preserved by this pair of maps and thus we have an embedding of $\L$-structures $A \to A'$.  
\end{proof}
\begin{rem}
    Using the proof of the above lemma, one could show an equivalence between the category $\K$ with $\L$-embeddings and the \emph{comma category} of the functors \[(\mathrm{FinBA},\mathrm{BA}\text{-}\mathrm{emb}) \xrightarrow[]{\mathrm{id}}(\mathrm{FinBA},\mathrm{BA}\text{-}\mathrm{hom}) \xleftarrow[]{\P} (\mathrm{FinSet}^{\mathrm{op}},\mathrm{injections})\]
    In particular, the data of an embedding of Boolean algebras $B(A) \to B(A')$ together with a surjective Boolean algebra homomorphism $\P(S(A')) \to \P(S(A)) $, not only guarantees the existence of an $\L$-embedding, but it characterizes such embedding. 
\end{rem}

\begin{prop}
$\K$ has $\mathrm{AP}$. 
\end{prop}
\begin{proof}
Consider the following amalgamation problem in $\K$: 
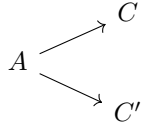
\begin{center}
  \begin{tikzcd}[row sep=tiny]
& C  \\
A \arrow[ur] \arrow[dr] & \\
& C'
\end{tikzcd}
\end{center}
By our previous lemma, this translates to the following diagram with commutative squares.

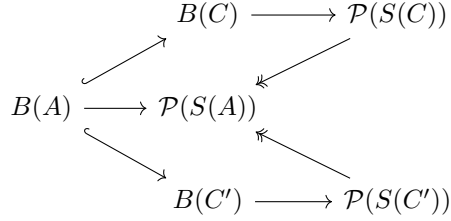
\begin{center}
  \begin{tikzcd}
& B(C) \arrow[r] & \mathcal{P}(S(C)) \arrow[two heads, dl]  \\
B(A) \arrow[hook,ur] \arrow[hook,dr] \arrow[r] &\mathcal{P}(S(A)) & \\
& B(C') \arrow[r] & \mathcal{P}(S(C')) \arrow[two heads, ul]
\end{tikzcd}
\end{center}
Applying Stone duality, we get the following diagram in the category of finite sets, maintaining commutativity.
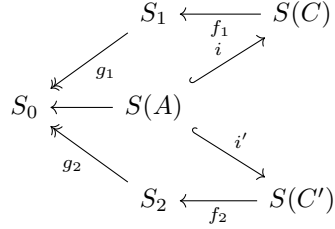
\begin{center}
  \begin{tikzcd}
&S_1 \arrow[dl,two heads,"g_1"] &\arrow[l,"f_1"] S(C) \\
S_0 &\arrow[l] S(A) \arrow[ur,hook,"i"] \arrow[dr,hook,"i'"]\\
& S_2 \arrow[ul, two heads,"g_2"] &\arrow[l,"f_2"] S(C')
\end{tikzcd}
\end{center}
Performing the pullback and pushout (in the category of finite sets) we get
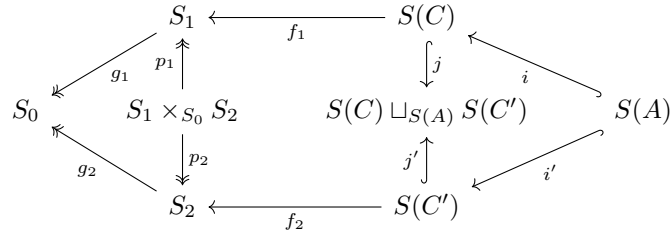
\begin{center}
  \begin{tikzcd}
&S_1 \arrow[dl,two heads,"g_1"] &\arrow[l,"f_1"] S(C) \arrow[d,hook,"j"]\\
S_0 & S_1 \times_{S_0}S_2 \arrow[u,two heads,"p_1"] \arrow[d,two heads,"p_2"] & S(C)  \sqcup_{S(A)}S(C')& \arrow[ul,hook,"i"] 
 \arrow[dl,hook,"i'"] S(A)\\
 &S_2 \arrow[ul,two heads,"g_2"]& \arrow[l,"f_2"] \arrow[u,hook,"j'"]S(C') 

\end{tikzcd}
\end{center}

Now, $S_1 \times_{S_0}S_2 = \{(a,b) \in S_1 \times S_2 \: g_1(a) = g_2(b)\}$ and the functions $p_i \colon S_1 \times_{S_0} S_2 \to S_i$ are the projections. Without loss of generality, we consider $i,i',j,j'$ inclusions of sets. We want to define a function 
\[h\colon S(C)\sqcup_{S(A)}S(C') \to S_1 \times_{S_0}S_2\]
such that $f_1 = p_1 \circ h \circ j$ and $f_2 = p_2 \circ h \circ j'$. Given $x \in S(A)$, define $h(x) = (f_1(x),f_2(x))$; since the outer paths are the same, $(f_1(x),f_2(x)) \in S_1 \times_{S_0}S_2$. Now, given $x \in S(C)$, since $p_1$ is surjective, pick $b \in S_2$ such that $(f_1(x),b) \in S_1 \times_{S_0}S_2$ and define $h(x) = (f_1(x),b)$. Similarly, for $x \in S(C')$, since $p_2$ is surjective, define $h(x) = (a,f_2(x))$ for some $a \in S_1$ such that $(a,f_2(x)) \in S_1 \times_{S_0}S_2$. It is easy to check that $f_1 = p_1 \circ h \circ j$ and $f_2 = p_2 \circ h \circ j'$. 

Dually, we get a Boolean algebra $B(D) = \mathcal{P}(S_1 \times_{S_0}S_2)$ that solves the amalgamation problem $B(A) \to B(C)$ and $B(A) \to B(C')$ and a set $S(D)$ with a homomorphism $B(D) \to \mathcal{P}(S(D))$. Using the lemma and commutativity of squares, we get a solution $(B(D),S(D))$ to the initial amalgamation problem. In conclusion, $\K$ has $\mathrm{AP}$.
\end{proof}

Let $T$ be the first-order theory of the Fra\"issé limit of $\K$.
We show that $T$ is $\SM$. 
By Fact \ref{1^n_prop} and Proposition \ref{Coopermax_IFF_SM}, it is enough to show that $T$ has $1^{(1)}$.
Let $\fhi(x;y) = R(x,y)$. Fix $n<\omega$ and consider the finite model $A$ of $T_0$ given by $B(A) = \mathcal{P}(n)$ and $S(B) = n$ with $R(i,X)$ if and only if $i \in X$. Now, for each $i<n$, define $b_i = b_{\{i\}}$. Clearly, $R(A,b_i) \cap R(A,b_j) = R(n,b_i)\cap R(n,b_j) = \{i\} \cap \{j\} = \emptyset$ for all $i<j<n$. Moreover, given $X \subseteq n$, we have that 
\[R(A,b_X) = R(n,b_X) = X = \bigcup_{i \in X}\{i\} = \bigcup_{i \in X}R(A,b_i)\]
Embedding $A$ into $\M$ (by universality of $\M$), the above properties are preserved. Thus $R(x,y)$ is $\SM$ in $T$.  

We can give an explicit description of the Fra\"issé limit as follows: 
Let $\M = (S(\M),B(\M)) =  (\omega,\B)$ where $\B$ is the countable atomless Boolean algebra. By Stone duality, we know that there is an isomorphism of Boolean algebras $\B \xrightarrow{g} \mathrm{CLOP}(2^\omega)$. 
Moreover, let $f\colon \omega \to 2^\omega$ such that the image of $f$ is a countable dense subset of the Cantor space and define $R(x,y)$ as follows: 
\[\M \models R(n,b) \Iff f(n) \in g(b)\]
Note that, since $2^\omega$ is totally disconnected, for every $n \neq m$ in $\omega$ there exists $b_{n}^m \in \B$ such that $\M \models R(n,b) \land \neg R(m,b)$.
Using this, one can check that the extension properties hold in $\M$ and thus show that it is the desired Fra\"issé limit.

\end{eg}

\section{Weakenings of \texorpdfstring{$\SM$}{TEXT}} \label{secWeak}
Now that we have established the worst (in some sense) combinatorial binary property that a complete first-order theory can have, we identify subcollections of the set of exhibitable patterns which are enough to describe some of the properties of formulas previously considered. From that, we define other notions of maximality realizing the patterns in those subcollections. 
\subsection{Positive maximality}
The \emph{tree properties} defining simple, $\NTP_1$ and $\NTP_2$ theories can be described by patterns in which the negation of the formula is not involved. More precisely, we can characterize tree properties by exhibiting patterns in which, for every condition $(X^+,X^-) \in \C \cup \I$, the set $X^-$ is empty (see Proposition \ref{div-lines-as-patterns}).
This observation leads the way for the concepts and results described in this subsection. 
\begin{defn} \label{positive-patt-defn}
    Fix $n<\omega$. A \emph{positive} $n$-pattern is a pattern $(\C,\I)$ such that for every condition $(X^+,X^-) \in \C \cup \I$, we have that $X^- = \emptyset$.  
\end{defn}
We can now define the notion of \emph{positive maximality} ($\PM$) by requesting the exhibition of all reasonable (See Definition \ref{Reasonable}) positive patterns. 
\begin{defn}\label{PM-defn}
    A formula $\fhi(x;y)$ is \emph{positive maximal} ($\PM$), if for every $n<\omega$ and every reasonable \emph{positive} $n$-pattern $(\C,\I)$, the formula $\fhi$ exhibits $(\C,\I)$. 
\end{defn}
\begin{rem}\label{reas+pos->exhib}
Notice that we didn't  require (as in $\SM$) to realize every exhibitable pattern since every reasonable positive pattern is exhibitable. Let's show this by giving an example of a $\PM$ theory. 
    Let $T = \mathrm{Th}(\mathcal{B})$ where $\mathcal{B}$ is an infinite Boolean algebra with an atomless element. We claim that the formula $\fhi(x;y) = 0 \neq x \leq y$ is $\PM$.
    The existence of an atomless element implies the existence of a sequence $(a_i)_{i<\omega}$ of non-zero elements below the atomless element such that $a_i \cap a_j = 0$ for every $i<j<\omega$. Fix $n<\omega$ and fix a reasonable positive $n$-pattern $(\C,\I)$. If $\C = \emptyset$, take $b_i = 0$ for every $i<n$; since $\fhi(x;b_i) = \bot$ for every $i$, we get that every inconsistency condition will be satisfied. 
    Assume $\emptyset \neq \C = \{(Y_j,\emptyset) \: j <k\} $. For each $i<n$, define
    \[b_i = \bigcup_{i \in Y_j}a_j\]
    Fix $\Tilde{j}<k$. Then, if $i \in Y_{\Tilde{j}}$ we have that 
    \[0\neq a_{\Tilde{j}} \leq \bigcup_{i \in Y_j}a_j = b_i\]
    hence $\{\fhi(x;b_i) \: i \in Y_{\Bar{j}}\}$ is consistent.
    Fix now $(Z,\emptyset) \in \I$ and suppose, toward contradiction, that there is $a \in \mathcal{B}$ such that for every $i \in Z$, $a \models \fhi(x;b_i)$.
    This means that 
    \[0 \neq a \leq \bigcap_{i \in Z}\bigcup_{i \in Y_j}a_j = \bigcup_{\substack{f\colon Z \to k,\\ (\forall i \in Z)(i \in Y_{f(i)}) }}\bigcap_{i \in Z} a_{f(i)}\]
    If the above union is $0 \in \mathcal{B}$, then $0 \neq a \leq 0$; contradiction. 
    If it is not empty, this means that there is $f\colon Z \to k$, with for all $i\in Z, i \in Y_{f(i)}$, such that $\bigcap_{i \in Z}a_{f(i)} \neq 0$. Since all the $a_j$ are pairwise disjoint, this can only happen if $f$ is constant, hence there is a $j<k$ such that for all $i \in Z$, $a_j$ appears in the above union i.e.\ there is a $j<k$ such that for all $i\in Z$, $i \in Y_j$; thus $Z \subseteq Y_j$ contradicting reasonableness.  
    In conclusion, this theory is $\PM$ (despite being $\NSM$, see Remark \ref{Atomless-not-SM}) and hence every reasonable positive $n$-pattern is exhibitable. 
\end{rem}
\begin{rem}
 Since every reasonable positive pattern is, in particular, an exhibitable pattern, if a formula $\fhi(x;y)$ is $\SM$, then $\fhi(x;y)$ is also $\PM$.
\end{rem}
We now prove a characterization of $\PM$ at the level of formulas. This equivalence will give us a concrete subfamily of the collection of all reasonable positive patterns which is enough to exhibit positive maximality. 
\begin{prop}
\label{PMchar}
    A formula $\fhi(x;y)$ is $\PM$ if and only if, for every $n<\omega$, there is $(b_X)_{X \subseteq n}$ such that, for any family of subsets of $n$, $\emptyset \neq Z \in \mathcal{P}(\mathcal{P}(n))$, the partial $\fhi$-type
    \[\{\fhi(x;b_Y) \: Y \in Z\} \text{ is consistent if and only if }\bigcap Z \neq \emptyset\]
\end{prop}
\begin{proof}
    Assume $\fhi(x;y)$ is $\PM$. Fix $n<\omega$ and consider the following reasonable positive $2^n$-pattern
    \[\C = \{(Z,\emptyset) \: Z \subseteq 2^n \text{ such that } \bigcap Z \neq \emptyset\}\]
    \[\I = \{(Z,\emptyset) \: Z \subseteq 2^n \text{ such that }\bigcap Z = \emptyset\}\]
    where we identify each $Z \subseteq 2^n$ as a family of subsets of $n$.
    By $\PM$, there is $(b_i)_{i<2^n} = (b_X)_{X \subseteq n}$ such that, for every $(Z,\emptyset) \in \C$, i.e.\ for every $Z \subseteq 2^n$ such that $\bigcap Z \neq \emptyset$, the type $\{\fhi(x;b_i) \: i \in Z\} = \{\fhi(x;b_Y) \: Y \in Z\}$ is consistent and, for every $(Z, \emptyset) \in \I$ i.e.\ for every $Z \subseteq 2^n$ such that $\bigcap Z = \emptyset$, the type $\{\fhi(x;b_Y) \: Y \in Z \}$ is inconsistent. 

    Assume now that $\fhi(x;y)$ has the above property; fix $n<\omega$ and let $(\C,\I)$ be a reasonable positive $n$-pattern. By the above property, let $(b_X)_{X \subseteq n}$ as above. We can assume that $\C \neq \emptyset$ (otherwise, take $b_i = b_\emptyset$ for all $i<n$). Hence let $\C = \{(Y_i,\emptyset) \: i <k\}$. For any $i<n$ define
    \[b_i = b_{\{j<k \: i \in Y_j\}}\]
    Fix $(Y_{\Tilde{j}},\emptyset) \in \C$. To show that $\{\fhi(x;b_i) \: i \in Y_{\Tilde{j}}\}$ is consistent, by the above property, it is enough to show that 
    \[\bigcap_{i \in Y_{\Tilde{j}}}\{j<k \: i \in Y_j\} \neq \emptyset\]
     For every $i \in Y_{\Tilde{j}}$, $\Tilde{j}$ is one of the $j<k$ such that $i \in Y_{\Tilde{j}}$, thus $\Tilde{j} \in \bigcap_{i \in Y_{\Tilde{j}}}\{j<k \: i \in Y_j\}$.

    Assume now that $(Z, \emptyset) \in \I$ and that, toward contradiction, there is $a \in \U^x$ such that for all $i \in Z$, we have that $a \models \fhi(x;b_i)$ i.e.\ 
    \[\{\fhi(x;b_i) \: i \in Z\} = \{\fhi(x;b_{\{j<k \: i \in Y_j\}}) \: i \in Z\} \text{ is consistent.}\]
    By the above property, this is equivalent to 
    \[\emptyset \neq \bigcap_{i \in Z}\{j<k \: i \in Y_j\} \]
    Take $j<k$ in the above intersection. Then, for every $i \in Z$, $i \in Y_j$ contradicting reasonability of the pattern. 
\end{proof}
\begin{rem}
In \cite{cooper}, Cooper defined a property of formulas (named \emph{vp}) that closely resembles the above characterization of $\PM$. It is easy to check that, using Proposition \ref{PMchar}, $\PM$ is equivalent to Cooper's \emph{vp}, even at the level of formulas.
\end{rem}

As pointed out above, some dividing lines can be defined through realizations of \emph{reasonable positive patterns}. In particular, $\TP_2$ is given by positive patterns and hence $\PM$ implies $\TP_2$ and a fortiori $\IP$ and $\OP$. 
As described above, $\SOP$ is not given by positive patterns; hence, a natural question is the following: Is there a $\NSOP$ and $\PM$ theory? 

We answer positively to the above question showing an example:
\begin{eg}[A positive maximal $\NSOP_{4}$ theory]\label{pm-nsop-example}
Let $\L = \{P,O,R,(E_k)_{0<k<\omega}\}$ where $\mathrm{ar}(P) = \mathrm{ar}(O) = 1$, $\mathrm{ar}(R) = 2$ and for any $k<\omega$, $\mathrm{ar}(E_k) = k$.\\
Let $T_0$ be the following $\L$-theory: 
\begin{itemize}
    \item[$(A_1)$] $P,O$ partition the universe. 
    \item[$(A_2^{k})$]$E_k \subseteq O^k$ is irreflexive.
    \item[$(A_3)$] $R \subseteq P \times O$
    \item[$(A_4^{k})$] $\forall y_0...y_{k-1}\forall x(E_k(\Bar{y})  \to \neg \bigwedge_{i<k}R(x,y_i) )$
\end{itemize}
Let $\K = \{A \: A \models T_0 \text{ and $A$ is finite}\}$. We claim that $\K$ is a Fra\"issé class with free amalgamation. 

Since $T_0$ is $\forall$-axiomatized, the class of models of $T_0$ (in particular the class of its finite models) is closed under taking substructures. Thus $\K$ has $\mathrm{HP}$. 
Since we allow the empty structure to be a model of $T_0$, it's enough to prove $\mathrm{AP}$ to get $\mathrm{JEP}$ for free. 

\begin{prop}
    \label{K-has-AP}
    The class $\K$ has the amalgamation property. 
\end{prop}
\begin{proof}
    Consider the amalgamation problem $(A,B_i,f_i)_{i<2}$ in $\U$ where $f_i\colon A \to B_i$ are embeddings.
    
    We define an $\L$-structure $C$ as follows: 
    The universe of $C$ is partitioned by $P^C = (P^{B_0} \setminus P^A) \sqcup P^{B_1}$ and $O^C = (O^{B_0} \setminus O^A) \sqcup O^{B_1}$. The interpretation of $R$ and $E_k$ is the induced one (with no other relations holding). In particular there is no relation $E_k$ holding between elements of $O^{B_0}$ and $O^{B_1}$ inside $C$.\\
    Clearly $C$ solves the amalgamation problem. We need to show that $C \models T_0$.\\
    Clearly $C$ satisfies the axioms (and axiom schemes) $A_1$, $(A_2^k)_{k<\omega}$ and $(A_3^k)_{k<\omega}$.
    It remains to show that, for each $k<\omega$, $C \models (A_4^k)$. So, fix $k<\omega$, fix $c_0...c_{k-1} \in O^C$ and $a \in P^C$. Moreover, assume that $E_k(\Bar{c})$ holds. We need to show that $a \models \neg \bigwedge_{i<k}R(x;c_i)$. Since $E_k$ is induced, we have that either $\Bar{c} \in (O^{B_0})^k$ or $\Bar{c} \in (O^{B_1})^k$, but not both. Without loss of generality, $\Bar{c} \in (O^{B_1})^k$. Now, either $a \in P^{B_0} \setminus P^A$ xor $a \in P^{B_1}$.
    In the first case, since $\Bar{c} \in P^{B_1}$ and $R$ is the induced one, we have that $\models \neg R(a,c_i)$ for every $i<k$ and hence $a \models \neg \bigwedge_{i<k}R(x;c_i)$. In the second case, since $B_1 \models T_0$, $a \models \neg \bigwedge_{i<k}R(x;c_i)$. In conclusion $C \models (A_4^k)$ and hence $\K$ has $AP$. 
\end{proof}
Note that, in the above proof, $C$ is the free amalgam of $B_0$ and $B_1$ over $A$. Hence $\K$ is closed under free amalgamation (We refer to \cite{free-amalgamation_Conant}, Definition 3.1 and Example 3.2). 

Since $\K$ is a Fra\"issé class, let $\M$ be its limit and let $T_\M = \mathrm{Th}(\M)$.
Since $\K$ is closed under free amalgamation, we have that $T_\M$ is $\NSOP_4$ (See \cite{free-amalgamation_Conant}, Theorem 1.1) and in particular $\NSOP$.
\begin{prop}
    $T_\M$ is $\PM$.
\end{prop}
\begin{proof}
We claim that the formula $R(x;y)$ is $\PM$.
Let $n<\omega$ and fix $\mathcal{P} = (\C,\I)$ a reasonable positive $n$-pattern, where \[\C = \{(Y_0,\emptyset),...,(Y_{l-1},\emptyset)\} \text{ and } \I = \{(Z_0,\emptyset),...,(Z_{h-1},\emptyset)\}\]

Define the structure $A_\mathcal{P}$ as follows:
The universe of $A_\mathcal{P}$ is given by $P^{A_\mathcal{P}} \cup O^{A_\mathcal{P}} = \{a_i \: i < l\} \cup \{b_i \: i<n\}$.
We interpret $R$ by $R(a_i;b_j)$ holds if and only if $i \in Y_j$ and, for $k<\omega$, we define $E_k$ as follows: $ E_k(b_{i_0},...,b_{i_{k-1}})$ holds if and only if $k = |Z_j|$ for some $j < h $ and $\{i_0,...,i_{k-1}\} = Z_j$.
First of all, by definition of reasonable pattern (see Definition \ref{Reasonable}) there is no $(Z,\emptyset) \in \I$ such that $(Z,\emptyset) \subseteq (Y,\emptyset)$ (coordinatewise) for some $(Y,\emptyset) \in \C$. This ensures that the witness for $\{R(x,b_i) \: i \in Y\}$ doesn't conflict with the relation $E_{k}$ for $k<|Y|$ and the axiom $(A_4^k)$. 
Thus, we have that $A_\mathcal{P} \models T_0$. Hence there is an embedding of  $f\colon A_\mathcal{P} \to \M$. Let $B = \{f(b_i) \: i < n\}$. It's clear that $R(x;y)$ exhibits the pattern $\mathcal{P}$ as witnessed by $B$.  
\end{proof}
\end{eg}
As we showed, this example is actually a $\NSOP_4$ and $\PM$ theory. This means that no $\SOP_n$, $n\geq 4$, can be described by positive patterns. It is still open whether $\SOP_n$, for $n \geq 4$, can be described through (non-positive) patterns (even though, since $\SM$ implies $\SOP$ which in turn implies $\SOP_n$, we have $\SM$ implies every $\SOP_n$).

Again, we remark that $\SOP_3$ was characterized by Kaplan, Ramsey and Simon (\cite{kaplan2023generic}, Lemma 7.3) in a way that can be described using positive patterns. 

\subsection{Consistency maximality}
While tree properties can be characterized by positive patterns, some properties of formulas, such as $\OP$ and $\IP$, do not need any inconsistency conditions i.e.\ the patterns $(\C,\I)$ needed to get these properties have $\I = \emptyset$ (see Proposition \ref{div-lines-as-patterns}). It is thus natural to define the following property. 
\begin{defn}\label{CM-defn}
    We say that a formula $\fhi(x;y)$ is \emph{consistency maximal} ($\CM$) if, for every $n<\omega$ and every reasonable (consistency) $n$-pattern $(\C,\emptyset)$, $\fhi$ exhibits $(\C,\emptyset)$.
\end{defn}
\begin{rem}
    Note that we didn't ask for a $\CM$ formula to realize every \emph{exhibitable} pattern of the form $(\C,\emptyset)$ since every such pattern is exhibitable. To see this we give an example of a $\CM$ formula. 

    Consider the theory $T$ of the random graph. Recall that this is a countable graph characterized (up to isomorphism) by the following property:

    For any $A,B$ finite disjoint subsets of vertices, there exists a vertex $v$ such that 
    \[\models E(v,a) \text{ for every $a \in A$}\]
    \[\models \neg E(v,b) \text{ for every $b \in B$ }\]
    where $E$ is the edge relation.
    
    Let $\fhi(x;y) = E(x;y)$, fix $n<\omega$, a reasonable $n$-pattern $(\C,\emptyset)$ and $n$ distinct vertices $b_0,...,b_{n-1}$. Then, for every condition $(A^+,A^-) \in \C$, by the above property characterizing the random graph, there is $v$ such that 
    \[v \models \{\fhi(x;b_i) \: i \in A^+\} \cup \{\neg \fhi(x;b_i) \: i \in A^-\}\]
    Hence $\fhi(x;y)$ exhibits $(\C,\emptyset)$. In conclusion $\fhi(x;y)$ is $\CM$. 
\end{rem}
We prove that $\CM$ is a weakening of $\PM$. 
\begin{prop}\label{PMimpliesCM}
    If a formula $\fhi(x;y)$ is $\PM$ then $\fhi(x;y)$ is $\CM$.
\end{prop}
\begin{proof}
    Let $n<\omega$ and let $(\C,\emptyset)$ be a reasonable $n$-pattern. Define the positive $2n$-pattern
    \[\C' := \{(A^+ \cup (n+A^-), \emptyset)\: (A^+,A^-)\in \C\}\]
    \[\I' = \{(\{i,i+n\},\emptyset) \: i<n\}\]
    where $n+A^- = \{n+i \: i \in A^-\}$. Note that, since $A^+ \cap A^- = \emptyset$, the positive pattern $(\C',\I')$ is reasonable.
    Since $\fhi(x;y)$ is positive maximal, there are $b_0,...,b_{n-1},b_{n},...,b_{2n-1}$ witnessing that $\fhi$ exhibits $(\C',\I')$.
    We claim that $b_0,...,b_{n-1}$ witness that $\fhi$ exhibits $(\C,\emptyset)$.
    Let $(A^+,A^-) \in \C$. Then 
    \[\{\fhi(x;b_i) \: i \in A^+ \cup (n+A^-)\} \text{ is consistent.}\]
    Let $a$ be a realization of it and let $i \in A^-$. If, toward contradiction, $a \models \fhi(x;b_i)$, then, by the choice of $a$, $a \models \fhi(x;b_i)\land \fhi(x;b_{n+i})$ which is inconsistent. Hence, for each $i \in A^-$, $a \models \neg \fhi(x;b_i)$. This shows that 
    \[\{\fhi(x;b_i) \: i \in A^+\} \cup \{\neg \fhi(x;b_j) \: j \in A^-\} \text{ is consistent.}\]
    Hence, $\fhi$ exhibits $(\C,\emptyset)$. In conclusion $\fhi$ is $\CM$. 
\end{proof}
We now give an example of a $\NPM$ but $\CM$ theory, showing that the implication in Proposition \ref{PMimpliesCM} is strict. 
\begin{eg}
    Again, consider the theory $T$ of the random graph. It is well known that this theory is simple ($\NTP$); moreover we proved that it is $\CM$. If $T$ was $\PM$, then, since $\TP$ can be realized through patterns, $T$ would have $\TP$, contradicting the simplicity of the theory.
\end{eg}

We now prove that this weaker notion of maximality ($\CM$) is actually equivalent to the \emph{Independence Property} $(\IP)$. 
Recall that a formula $\fhi(x;y)$ has $\IP$ if, by definition, there are tuples $(a_i)_{i<\omega}$ and $(b_I)_{I \subseteq \omega}$ such that 
\[\models \fhi(a_i;b_I) \iff i \in I \]
By compactness, this is equivalent to having, for every $n<\omega$, a $n$th approximation of the above property. 
\begin{prop}\label{CMiffIP}
    A formula $\fhi(x;y)$ is $\CM$ if and only if it has $\IP$
\end{prop}
\begin{proof}
    We will actually prove that $\fhi(x;y)$ is $\CM$ if and only if $\fhi^{\mathrm{opp}}(y;x)$ has $\IP$. On the other hand, it's well-known that a formula has $\IP$ if and only if the opposite formula has $\IP$ (See, e.g., \cite{casanovas-NIP}).  

    Suppose $\fhi(x;y)$ is $\CM$. Fix $n<\omega$ and consider the pattern 
    \[\C = \{(X,n\setminus X) \: X \in \mathcal{P}(n)\}\]
    \[\I = \emptyset\]
    Since this is a consistency pattern, $\fhi(x,y)$ exhibits it; let $B=(b_i)_{i<n}$ be a witness. For any $X \in \P(n) = \C$, let $a_X$ a realization of $p_{X,n\setminus X}^B(x)$. Then,
    \[\models \fhi^{\mathrm{opp}}(b_i,a_X) \iff \models\fhi(a_X,b_i) \iff i \in X\]
    Hence, for every $n<\omega$, $\fhi^{\mathrm{opp}}(y;x)$ has a $n$th approximation of $\IP$. By compactness, $\fhi^{\mathrm{opp}}(y;x)$ has $\IP$. 

    Assume now that $\fhi^{\mathrm{opp}}(y;x)$ has $\IP$. Take $(b_i)_{i<\omega}$ and $(a_I)_{I \subseteq \omega}$ witnessing it. 
    Fix a reasonable $n$-pattern $(\C,\emptyset)$. Then, for any condition $(A^+,A^-) \in \C$, we have that 
    
    \[a_{A^+} \models \{\fhi^{\mathrm{opp}}(b_i;x) \: i \in A^+\} \cup \{\fhi^{\mathrm{opp}}(b_i;x) \: i \notin A^+, i<n\}\]
    Since $A^- \subseteq n \setminus A^+$, we have 
    
    \[ \{\fhi^{\mathrm{opp}}(b_i;x) \: i \in A^+\} \cup \{\fhi^{\mathrm{opp}}(b_i;x) \: i \notin A^+\} \supseteq \]\[ \{\fhi^{\mathrm{opp}}(b_i;x) \: i \in A^+\} \cup \{\fhi^{\mathrm{opp}}(b_i;x) \: i \in A^-\} \] Which in turn equals to \[\{\fhi(x;b_i) \: i \in A^+\} \cup \{\fhi(x;b_i) \: i \in A^-\} \]
    Hence $p_{A^+,A^-}^B(x)$ is consistent and hence $\fhi(x;y)$ is $\CM$. 
\end{proof}

\begin{rem}
    Given the above equivalence, we have an alternative proof of Proposition $\ref{PMimpliesCM}$:

    Since $\TP_2$ can be realized by positive patterns, if $\fhi(x;y)$ is $\PM$, then $\fhi(x;y)$ has $\TP_2$ which in turn implies $\IP$ (See Fact \ref{implications-classical-div-lines}) which, by Proposition \ref{CMiffIP}, is equivalent to $\CM$.
\end{rem}

\section{The \texorpdfstring{$\PM^{(k)}$}{TEXT} hierarchy}\label{Sec-PMk}
Natural weakenings of $\PM$ can be obtained by restricting the size of possible sets of inconsistencies. Some combinatorial properties that can be defined in terms of inconsistency sets of size $>2$, such as $k$-$\TP$ and $k$-$\TP_2$, turn out to be equivalent, at the level of theories, to their counterparts with inconsistency sets of size $2$ ($2$-$\TP$ and $2$-$\TP_2$). In contrast, by weakening $\PM$ restricting the size of possible sets of inconsistencies, we get a \emph{strict} hierarchy of properties $(\PM^{(k)})_{1<k<\omega}$. 

\begin{defn}\label{PM_k-defn}
    Let $1<k<\omega$. A formula $\fhi(x;y)$ is $\PM^{(k)}$ if for all $n<\omega$ and for every reasonable positive $n$-pattern $(\C,\I)$ with for all $(Z^+,\emptyset) \in \I$, $|Z^+| = k$, $\fhi(x;y)$ exhibits $(\C,\I)$. 
    We say that a theory $T$ is $\PM^{(k)}$ if there is a $\PM^{(k)}$ formula. Otherwise we say $T$ is $\NPM^{(k)}$.
\end{defn}
\begin{rem}
    We don't consider $\PM^{(0)}$ and $\PM^{(1)}$ since every theory is $\PM^{(0)}$ and $\PM^{(1)}$. For $\PM^{(0)}$, fix a tuple $a \in \U^x$ and take the formula 
    \[\fhi(x;y) = (x=y)\]
    Since we excluded conditions of the form $(\emptyset,\emptyset)$ from the definition of patterns, $\I = \emptyset$ and thus $\fhi(x;y)$ is asked to exhibit positive patterns of the form $(\C,\emptyset)$. Fix $n<\omega$ and an $n$-pattern as above.  Define $b_i =a $ for all $i<n$; then, for any condition $(A^+,\emptyset) \in \C$
    \[\bigwedge_{i \in A^+} (x=a) \text{ is consistent, witnessed by $a$}\]
    For $\PM^{(1)}$, fix $a\neq b$ and consider the formula 
    \[\fhi(x;y) = (y = a) \land (x=x)\]
    Fix $n<\omega$ and a reasonable positive $n$-pattern $(\C,\I)$ with for all $(Z^+,\emptyset)$, $|Z^+| = 1$. For all $i$ such that there is a condition $(\{i\},\emptyset) \in \I$, define $b_i = b$. Otherwise, define $b_i = a$. Then, for $(A^+,\emptyset) \in \C$, since pattern are reasonable, for all $i \in A^+$, $b_i = a$ and thus 
    \[\bigwedge_{i<n}(a=a)\land (x=x) \text{ is consistent.} \]
    For $(\{i\},\emptyset) \in \I$, $b_i = b \neq a$ and thus 
    \[\fhi(x;b_i) = (a=b) \land (x=x) \text{ is inconsistent.}\]
    In conclusion, we disregard $\PM^{(0)}$ and $\PM^{(1)}$ since they are not very sensible notions. 
\end{rem}
We can characterize the $\PM^{(k)}$ properties by ``realizations of hypergraphs":
\begin{prop}\label{PMk-real-hypergraph}
    $\fhi(x;y)$ is $\PM^{(k)}$ if and only if for every finite $k$-hypergraph $(n,E)$, $\fhi$ \emph{realizes} $(n,E)$ i.e.\ there are $(b_i)_{i<n}$ such that 
    \begin{itemize}
        \item[(i)] for any $i_0,...,i_{k-1} \in n$ with $\neg E(i_0,...,i_{k-1})$, we have \[\{\fhi(x;b_{i_j}) \: j<k\} \text{ is inconsistent}\]
        and
        \item[(ii)] for any clique $\{i_0,...,i_{l-1}\} \subseteq n$
        \[\{\fhi(x;b_{i_j}) \: j<l\} \text{ is consistent.}\]
    \end{itemize}
\end{prop}
\begin{proof}
    Suppose $\fhi(x;y)$ is $\PM^{(k)}$. Let $(n,E)$ a finite $k$-hypergraph. Consider the positive $n$-pattern $(\C,\I)$ where 
    \[\C := \{(\{i_0,...,i_{l-1}\},\emptyset) \: \{i_0,...,i_{l-1}\}\text{ is a clique }\}\]
    \[\I :=\{(\{i_0,...,i_{k-1}\},\emptyset) \: \neg E(i_0,...,i_{k-1})\}\]
    Since no non-hyperedge can be a subset of a clique, this is a reasonable pattern. By $\PM^{(k)}$, $\fhi$ exhibits $(\C,\I)$ and thus there are $(b_i)_{i<n}$ witnessing it. It's easy to see that, with these parameters, $\fhi$ realizes $(n,E)$.

    Assume now that $\fhi(x;y)$ realizes every $k$-hypergraph. Let $n<\omega$ and let $(\C,\I)$ a positive $n$-pattern with for all $(Z^+,\emptyset) \in \I$, $|Z^+| = k$. Consider the graph $(n,E)$ where, $E(i_0,...,i_{k-1})$ holds if and only if $(\{i_0,...,i_{k-1}\},\emptyset) \notin \I$. Since $\fhi(x;y)$ realizes every finite $k$-hypergraph, there are $(b_i)_{i<n}$ witnessing it.
    Now, given $(A^+,\emptyset) \in \C$, since the pattern is reasonable, every $k$-element subset $B$ of $A^+$ is such that $(B,\emptyset) \notin \I$, thus $A^+$ forms a clique in $(n,E)$ and therefore 
    \[\{\fhi(x;b_i) \: i \in A^+\} \text{ is consistent.}\]
    On the other hand, given $(\{i_0,...,i_{k-1}\},\emptyset) \in \I$, we have  $\neg E(i_0,...,i_{k-1})$ and thus 
    \[\{\fhi(x;b_{i_j}) \: j<k\} \text{ is inconsistent.}\]
    Thus $\fhi$ exhibits $(\C,\I)$. 
\end{proof}
\begin{prop}
    If $\fhi(x;y)$ is $\PM^{(k)}$, then $\fhi(x;y)$ realizes the generic $k$-hypergraph. Moreover, the parameters may be chosen to be (generic ordered $k$-hypergraph)-indiscernibles.  
\end{prop}
\begin{proof}
    Apply compactness and the modeling property (which ordered $k$-hypergraphs have by Fact \ref{gener-ind-hypergraph}).
\end{proof}
\begin{prop}\label{PM_k+1-implies_PM_k}
    For every $1<k<\omega$, if $\fhi(x;y)$ is $\PM^{(k+1)}$, then $\psi(x;y_0...y_{k}) = \bigwedge_{i<k+1}\fhi(x;y_i)$ is $\PM^{(k)}$.
\end{prop}
\begin{proof}
    Assume $\fhi(x;y)$ has $\PM^{(k+1)}$. Consider the formula 
    \[\psi(x;y_0...y_{k}) = \bigwedge_{i<k+1}\fhi(x;y_i)\]
    We want to show that $\psi$ is $\PM^{(k)}$.
    Let $n<\omega$ and let $(n,E)$ be a finite $k$-hypergraph. Define a $(k+1)$-hypergraph  $(\Tilde{V},\Tilde{E})$ as follows: 
    \begin{itemize}
        \item[(i)] for each $i < n$, take $k+1$ new vertices $K_i = v_i^0,...,v_i^{k}$ with $\Tilde{E}(v_i^0,...,v_i^k)$.
        \item[(ii)] for each hyperedge $i_0,...,i_{k-1}$ in $(n,E)$, make $K_{i_0}...K_{i_{k-1}}$ into a $\Tilde{E}$-clique. 
        \item[(iii)] No other hyperedge is present.  
    \end{itemize}
    Now, $(\Tilde{V},\Tilde{E})$ is a $(k+1)$-hypergraph, thus by $\PM^{(k+1)}$, there are $(b_v)_{v \in \Tilde{V}}$ realizing it.
    Let ${i_0,...,i_{k-1}}$ a non-hyperedge in $(n,E)$. This implies that $(K_{i_j})_{j<k}$ doesn't form a $\Tilde{E}$-clique and hence 
    $\{\fhi(x;b_v) \: v \in \bigcup_{j<k}K_{i_j}\}$ is inconsistent. Thus 
    \[\{\psi(x;\bar{b}_{K_{i_j}}) \: j<k\} \text{ is inconsistent.}\]
    If $i_0,...,i_{l-1}$ is a clique in $(n,E)$, then $(K_{i_j})_{j<l}$ forms a clique in $(\Tilde{V},\Tilde{E})$ and thus $\{\fhi(x;b_v) \: v \in \bigcup_{j<l}K_{i_j}\}$ is consistent; this shows that
    \[\{\fhi(x;\bar{b}_{K_j}) \: j<l\}\]
    is consistent. 
    In conclusion $\psi(x;y_0...y_{k})$ realizes $(n,E)$ and hence $\psi$ is $\PM^{(k)}$. 
\end{proof}
\begin{defn}
    A theory is 
    \begin{itemize}
        \item[(i)] $\PM^{(\infty)}$ if there are formulae $(\fhi_k(x_k;y_k))_{1<k<\omega}$ such that for all $1<k<\omega$ $\fhi_k(x_k;y_k)$ is $\PM^{(k)}$.
        \item[(ii)] $UPM^{(\infty)}$ (Uniformly $\PM^{(\infty)}$) if there is a formula $\fhi(x;y)$ with $\PM^{(k)}$ for all $1<k<\omega$.
    \end{itemize}
\end{defn}
\begin{rem}
    Clearly, at the level of theories, for all $2<k<\omega$ we have
\[\PM \Rightarrow \UPM^{(\infty)} \Rightarrow \PM^{(\infty)} \Rightarrow ... \Rightarrow \PM^{(k+1)} \Rightarrow \PM^{(k)} \Rightarrow ... \Rightarrow \PM^{(2)}\]
We will show that the implication $\PM^{(k+1)} \Rightarrow \PM^{(k)}$ is strict for every $1<k<\omega$. 
\end{rem}

\begin{quest}
    Is there a $\PM^{(\infty)}$ but $\NUPM^{(\infty)}$ theory? Similarly, is there a $\UPM^{(\infty)}$ but $\NPM$ theory?
\end{quest}

\begin{defn}\label{nary}
    Fix $n<\omega$.
    A relational language $\L$ is $n$-ary if, for every $R \in \L$, $\mathrm{ar}(R) \leq n$. 
     An $\L$-theory $T$ is $n$-ary if 
    \begin{itemize}
        \item[(i)] $\L$ is $n$-ary, and
        \item[(ii)] $T$ has $\mathrm{QE}$.  
    \end{itemize}
\end{defn}
\begin{rem}
    We could have defined a theory to be $n$-ary if it is interdefinable with a theory satisfying the conditions in \ref{nary}; on the other hand, all the properties that we are considering here are preserved under interdefinability and thus we don't actually lose any generality. 
\end{rem}
In \cite{n-dependence}, Chernikov, Palacin and Takeuchi studied higher arity versions of $\IP$, firstly introduced by Shelah, called $\IP_k$ ($0<k<\omega$). In particular, they characterized $\NIP_k$ in terms of the collapse of ordered $(k+1)$-hypergraphs indiscernibles to order-indiscernibles. Using this characterization, we get the following result.
\begin{thm}\label{PMk+1impliesIPk}
    For any $0<k<\omega$, $\PM^{(k+1)}$ implies $\IP_k$. 
\end{thm}
\begin{proof}
    Fix $0<k<\omega$. 
    By Theorem 5.4 in \cite{n-dependence}, it's enough to exhibit a (generic ordered $(k+1)$-hypergraph)-indiscernible that is not order indiscernible. 
    Suppose $\fhi(x;y)$ is $\PM^{(k+1)}$. By Ramsey and compactness, we can find $(b_h)_{h \in H}$ $(H,<,E)$-indiscernible witnessing it, where $(H,<,E)$ is the generic ordered $(k+1)$-hypergraph. Let $h_0<h_1<...<h_{k+1} \in H$ such that 
    \[\models E(h_0...h_k) \land \neg E(h_1...h_{k+1})\]
    By $\PM^{(k+1)}$ we have 
    \[\models \exists x \bigwedge_{i<k+1}\fhi(x;b_{h_i}) \land \neg \exists x \bigwedge_{0<i<k+2} \fhi(x;b_{h_i})\]
    and thus, even if the order type of $h_0<...<h_k$ is the same as the one of $h_1<...<h_{k+1}$, we have 
    \[\tp((b_{h_i})_{i<k+1}) \neq \tp((b_{h_i})_{0<i<k+2}\]
    In conclusion $(b_h)_{h \in H}$ is not order indiscernible and thus some formula has $\IP_k$. 
\end{proof}
\begin{proof}[Alternative proof]
Suppose $\fhi(x;y)$ is $\PM^{(k+1)}$. Then, by Proposition \ref{PMk-real-hypergraph} and compactness, the formula \[\psi(y_0,...,y_{k}) =\exists x \bigwedge_{i<k+1}\fhi(x;y_i)\] \emph{encodes} every $(k+1)$-hypergraph; in particular, it has $\IP_k$ (see \cite{n-dependence}, Proposition 5.2).  
    
\end{proof}
\begin{cor}\label{k-ary_are_NPMk+1}
    Any $k$-ary theory is $\NPM^{(k+1)}$.
\end{cor}
\begin{proof}
    From \cite{n-dependence}, Example 2.2, we have that $k$-ary theories are $\NIP_k$. By \ref{PMk+1impliesIPk}, $k$-ary theories are $\NPM^{(k+1)}$. 
\end{proof}
We now construct examples of theories  with $\PM^{(k)}$ but not $\PM^{(k+1)}$ for all $k\geq 2$. This will show that the implication $\PM^{(k+1)} \Rightarrow PM^{(k)}$ is strict. We give a uniform recipe to build $\PM^{(k)}$ $\NPM^{(k+1)}$ theories.

\begin{eg}[$\PM^{(k)}$ $\NPM^{(k+1)}$ theory]\label{PM_k,NPM_k+1-examples}
Fix $1<k<\omega$ and consider the language $\L_k = \{O,P,R,E_k\}$ where $P,O$ are unary predicates, $R$ is a binary predicate and $E_k$ is a $k$-ary predicate. Consider the $\L_k$-theory $T_0$ axiomatized as follows:
\begin{itemize}
    \item[$(A_1)$] $P,O$ partition the universe. 
    \item[$(A_2)$]$E_k \subseteq P^k$ and $(P,E_k)$ is a $k$-hypergraph.
    \item[$(A_3)$] $R \subseteq O \times P$
    \item[$(A_4)$] $\forall v_0...v_{k-1}\forall x(\bigwedge_{i<k}R(x,v_i) \to \neg E_k(\Bar{v}))$
\end{itemize}
Let $\K = \{A \models T_0\: |A| < \aleph_0 \}$. Similarly to the proof of Proposition \ref{K-has-AP}, it can be shown that $\K$ is a Fra\"issé class closed under free-amalgamation. Let $\M_k$ be the Fra\"issé limit. Intuitively $\M_k$ has two disjoint sets ($O$,$P$), $(P(\M_k),E_k)$ is the generic $k$-hypergraph and $O(\M_k)$ is the set of witnesses for the consistency along $R$ of independent sets and the inconsistency along $R$ of $k$-hyperedges. In particular $T_k:= \mathrm{Th}(\M_k)$ is $\PM^{(k)}$ (witnessed by $R$), $\omega$-categorical, $\NSOP_4$ (See \cite{free-amalgamation_Conant}, Theorem 1.1) and eliminates quantifiers. Moreover, by Corollary \ref{k-ary_are_NPMk+1}, $T_k$ is $\NPM^{(k+1)}$.
\end{eg}

We now consider a well-known theory and see how it relates with this hierarchy: 
\begin{eg}[Generic triangle-free graph]
Let $\L = \{e\}$ where $e$ is a binary relation and let  $\K$ be the class of finite triangle-free graphs. It can be shown that $\K$ is a Fra\"issé class; let $\H$ be its limit and let $T_\H = \mathrm{Th}(\H)$. 
\begin{prop}
    $T_\H$ is $\PM^{(2)}$. 
\end{prop}
\begin{proof}
    We claim that the formula $\fhi(x;yz) = e(x,y) \land e(x,z)$ is $\PM^{(2)}$. 
    Let $n<\omega$ and let $(V,E)$ any finite graph an $n$ vertices. Construct the graph $(\Tilde{V},\Tilde{E})$ as follows: Replace each vertex $v \in V$ with two vertices $b_vc_v$ with $\neg b_v \Tilde{E}c_v$. Whenever $\{v,w\} \notin E$, make a new edge between $\{b_v,c_w\} \in \Tilde{E}$ and no other edges. 
    \begin{claim}
        $(\Tilde{V},\Tilde{E})$ is triangle-free. 
    \end{claim}
    \begin{claimproof}
        As a notation, given a graph and a vertex $v$, we write $N(v)$ for the graph neighborhood of $v$ i.e.\ the set of all the vertices in the graph that are connected by an edge to $v$. 
        Suppose, toward contradiction, that there is a triangle $a \Tilde{E} a' \Tilde{E}a'' \Tilde{a} $. Since there's no edge between $b_v$ and $c_v$ for every $v$, there must be $v_0,v_1,v_2$, such that the triangle is among $(b_{v_i}c_{v_i})_{i<3}$. Suppose that $b_{v_0}$ is in the triangle. But $N(b_{v_0}) = \{c_{v_1},c_{v_2}\}$ and $\neg c_{v_1} \Tilde{E} c_{v_2}$, contradiction. Then it must be the case that $c_{v_0}$ is in the triangle. But $N(c_{v_0}) = \{b_{v_1},b_{v_2}\}$ and $\neg b_{v_1}\Tilde{E}b_{v_2}$. Contradiction. Thus $(\Tilde{V},\Tilde{E})$ is triangle-free. 
    \end{claimproof}
    
    Since $(\Tilde{V},\Tilde{E})$ is a finite triangle-free graph, it can be embedded in $\H$. Now, let $\{v,w\} \notin E$. Then, if there is a vertex $a$ such that 
    \[a \models \fhi(x;b_vc_v) \land \fhi(x;b_wc_w) = (e(x,b_v) \land e(x,c_v)) \land ((e(x,b_w) \land e(x,c_w))\]
    Then $a,b_v,c_w$ form a triangle; contradiction. On the other hand, if $v_0,...,v_{n-1}$ is a clique, then $(b_{v_i}c_{v_i})_{i<n}$ is an independent set of vertices and hence, by the extension property of $\H$, there is a vertex $a$ such that $a \models \bigwedge_{i<n} e(x,b_{v_i}) \land e(x,c_{v_i})$ and thus 
    \[\{\fhi(x;b_{v_i}c_{v_i}) \: i<n\} \text{ is consistent.}\]
    In conclusion $\fhi(x;y)$ is $\PM^{(2)}$. Thus, by definition, $T_\H$ is $\PM^{(2)}$. 
\end{proof}
Moreover, since $T_\H$ is a binary theory, by Corollary \ref{k-ary_are_NPMk+1}, $T_\H$ is $\NPM^{(3)}$. 
\end{eg}

\addtocontents{toc}{\protect\setcounter{tocdepth}{0}}
\section*{Acknowledgements}
This work is part of the author's PhD dissertation in progress.
The author would like to sincerely thank Alex Kruckman for his guidance, valuable help and fruitful discussions. Furthermore, the exciting environment of the Wesleyan University logic group played an important role in this project, in particular the encouragement and the helpful conversations the author had with Cameron Hill, Rehana Patel and Alex Van Abel. Finally, the author would like to thank the anonymous referee for their useful comments and corrections. 
\addtocontents{toc}{\protect\setcounter{tocdepth}{2}}
\bibliographystyle{plainnat}
\bibliography{refnew.bib}

\end{document}